\documentclass[11pt,reqno]{amsart}
\usepackage{amsmath,amssymb,amsthm,amsfonts}
\usepackage[margin=1.15in]{geometry}
\usepackage[colorlinks=true,linkcolor=blue,citecolor=blue,urlcolor=blue]{hyperref}

\newtheorem{theorem}{Theorem}[section]
\newtheorem{proposition}[theorem]{Proposition}
\newtheorem{lemma}[theorem]{Lemma}
\newtheorem{corollary}[theorem]{Corollary}
\theoremstyle{definition}
\newtheorem{definition}[theorem]{Definition}
\newtheorem{remark}[theorem]{Remark}
\newtheorem{convention}[theorem]{Convention}
\newtheorem{problem}[theorem]{Problem}
\newtheorem*{theoremA}{Theorem A}
\newtheorem*{theoremB}{Theorem B}
\newtheorem*{theoremC}{Theorem C}
\newtheorem*{conjectureY}{Yau's Conjecture}
\numberwithin{equation}{section}

\newcommand{\bS}{\mathbb{S}}
\newcommand{\bR}{\mathbb{R}}
\newcommand{\bC}{\mathbb{C}}
\newcommand{\bT}{\mathbb{T}}
\newcommand{\bZ}{\mathbb{Z}}
\newcommand{\bK}{\mathbb{K}}
\newcommand{\cG}{\mathcal{G}}
\newcommand{\cB}{\mathcal{B}}
\newcommand{\cT}{\mathcal{T}}
\newcommand{\cK}{\mathcal{K}}
\newcommand{\cJ}{\mathcal{J}}
\newcommand{\cR}{\mathcal{R}}
\newcommand{\cE}{\mathcal{E}}
\newcommand{\Ric}{\operatorname{Ric}}
\newcommand{\arcosh}{\operatorname{arcosh}}
\newcommand{\sech}{\operatorname{sech}}
\newcommand{\Line}{\operatorname{Line}}
\newcommand{\avg}{\operatorname{avg}}
\newcommand{\dd}{\,d}
\newcommand{\rx}{\mathrm{x}}
\newcommand{\ry}{\mathrm{y}}
\newcommand{\rz}{\mathrm{z}}
\newcommand{\Tfl}{\bT_{\mathrm{fl}}}

\begin{document}

\title[Yau's conjecture for stacked Clifford tori]
{Yau's conjecture for the stacked Clifford tori of Wiygul}
\author{Alexander Pigazzini}
\address{}
\email{}

\begin{abstract}
We prove Yau's conjecture $\lambda_{1}=2$ for the stacked Clifford tori of Wiygul: for all integers
$N\ge2$, $k,\ell\ge1$ and every sufficiently large $m$, every closed embedded minimal surface arising from Wiygul's construction, of genus $k\ell m^{2}(N-1)+1$ in the round three-sphere and resembling $N$ parallel copies of the Clifford torus joined by small catenoidal tunnels, has first Laplace eigenvalue $2$. For $N\ge3$ these surfaces are chains rather than doublings, and the even--odd decomposition on which all previous verifications for gluing constructions rest is not available. The reflection lemma of Choe and Soret reduces the
problem to the sector of functions invariant under the symmetry group of the construction, and we show that the lowest nonzero eigenvalue of that sector equals $4+O(m^{-1})$. The value $4$ is the outcome of an exact identity in the limiting weighted graph model: the limiting waist ratios of the
construction form the Perron vector of the adjacency operator of the line graph of a path, so that the spectral gap of the path cancels against the total conductance of the tunnels prescribed by the balancing conditions, and what survives is the coefficient of the Jacobi operator of the Clifford torus. The analytic input consists of a conformally invariant channel inequality on a cylinder and of a Poincar\'e inequality on a periodically perforated torus. The properties of the construction on which the argument rests are isolated in a reduction theorem for closed surfaces decomposed into blocks joined by families of thin channels along the edges of a finite graph, subject to a symmetry assumption and to Poincar\'e and trace inequalities on the blocks.
\end{abstract}

\maketitle
\tableofcontents

\section{Introduction}\label{sec:intro}

\subsection{Yau's conjecture}
Let $\Sigma^{n}$ be a closed minimal hypersurface of the round unit sphere $\bS^{n+1}\subset\bR^{n+2}$ and let $0=\lambda_{0}(\Sigma)<\lambda_{1}(\Sigma)\le\lambda_{2}(\Sigma)
\le\cdots$ be the spectrum of its Laplace--Beltrami operator. By a theorem of Takahashi \cite{Takahashi} the restrictions to $\Sigma$ of the coordinate functions $x_{1},\dots,x_{n+2}$ of
$\bR^{n+2}$ satisfy \begin{equation}\label{eq:takahashi}
\Delta_{\Sigma}x_{i}+n\,x_{i}=0 ,
\end{equation}
so that each $x_{i}|_{\Sigma}$ has vanishing mean and is admissible in the variational
characterisation of $\lambda_{1}$; hence $\lambda_{1}(\Sigma)\le n$. In his problem list Yau
\cite[Problem 100]{Yau} asked whether equality holds for embedded hypersurfaces.

\begin{conjectureY}
If $\Sigma^{n}$ is a closed embedded minimal hypersurface of $\bS^{n+1}$, then
$\lambda_{1}(\Sigma)=n$.
\end{conjectureY}

To our knowledge the conjecture is open in every dimension $n\ge2$. An affirmative answer implies genus-dependent area bounds for embedded minimal surfaces of $\bS^{3}$ \cite{YangYau} and, through the theorem of Montiel and Ros on minimal immersions by first eigenfunctions \cite{MontielRos},
Lawson's conjecture that the Clifford torus is the only embedded minimal torus of $\bS^{3}$, proved
by Brendle \cite{Brendle}. Two kinds of partial results are known. Choi and Wang \cite{ChoiWang}
proved the universal bound $\lambda_{1}\ge n/2$ by Reilly's formula on the two components of
$\bS^{n+1}\setminus\Sigma$, and quantitative refinements were obtained in \cite{DSS,JTZ}; the resulting universal lower bounds remain strictly below $n$. In the other direction, equality has been established on explicit families: Tang and Yan \cite{TangYan} for all minimal isoparametric hypersurfaces, Choe and Soret \cite{ChoeSoret} for the surfaces of Lawson \cite{Lawson} and of Karcher, Pinkall and Sterling \cite{KPS}, and Kapouleas and McGrath \cite{KM24} for the doublings of the equatorial two-sphere constructible by the linearised doubling theorem of \cite{KMcjm}, which develops
\cite{Kapouleas}, and for the doublings of the Clifford torus constructed in \cite{KM24}.

\subsection{Doublings and chains}\label{ss:chains}
All verifications for gluing constructions concern \emph{doublings}: two nearly parallel copies of
a base surface joined by small catenoidal tunnels. There, exactly or after a small perturbation of
the metric, an involutive isometry exchanges the two copies, and each eigenspace splits into an even
part, controlled by comparison with the base surface, and an odd part, controlled by the tunnels.
This dichotomy is the backbone of \cite{KM24}.

For $N\ge3$ copies arranged in a chain there is no involution whose orbits are the copies and no
such dichotomy. The functions which are constant on each copy form an $N$-dimensional space, and
their Rayleigh quotients are governed by the combinatorial Laplacian $L(P_{N})$ of the path graph on
$N$ vertices, whose spectral gap $2(1-\cos(\pi/N))$ decays like $\pi^{2}N^{-2}$. If the tunnels
joining consecutive copies had a conductance independent of $N$, the lowest such quotient would
tend to zero and long chains would violate Yau's conjecture. That this does not happen, and why, is
the content of this paper.

\subsection{Stacked Clifford tori}
Wiygul \cite{Wiygul}, extending the doubling construction of Kapouleas and Yang \cite{KY}, proved
that for all integers $N\ge2$, $k,\ell\ge1$ and every sufficiently large $m$ there is a closed
embedded minimal surface $\Sigma[N,k,\ell,m]\subset\bS^{3}$ of genus $k\ell m^{2}(N-1)+1$, invariant
under a group $\cG[k,\ell,m]\cong D_{km}\times D_{\ell m}$ of isometries, which resembles $N$ nested
small perturbations of a fixed Clifford torus $\bT$ joined by $k\ell m^{2}$ small catenoidal tunnels
between each pair of neighbouring tori, the tunnels of consecutive layers sitting over two
interlaced rectangular lattices; as $m\to\infty$ the surfaces converge to $\bT$ with multiplicity
$N$. The surface is obtained as a normal graph over an explicit initial surface, and every quantity
of the construction which enters the present argument (lattices, gluing radius, waist radii,
balancing conditions, conformal lengths of the tunnels) is explicit. In Section \ref{sec:surfaces}
we reproduce these data and we isolate the single property of the final surface which is used,
namely that its induced metric is uniformly close to that of the initial surface; we verify from the
final estimate of \cite{Wiygul} that the surfaces constructed there have this property. A closed
embedded minimal surface with these features is called a \emph{stacked Clifford torus of type
$(N,k,\ell,m)$} (Definition \ref{def:stacked}); all the surfaces of \cite{Wiygul} are of this kind
(Corollary \ref{cor:wiygulW}).

\subsection{Results}
Throughout, $\cG=\cG[k,\ell,m]$ is the symmetry group of the construction, and the
$\cG$-invariant sector of the Laplacian of $\Sigma$ is its restriction to $\cG$-invariant functions.

\begin{theoremA}
Let $N\ge2$ and $k,\ell\ge1$ be integers. There exists $m_{0}=m_{0}[N,k,\ell]$ such that for every
$m\ge m_{0}$ every stacked Clifford torus $\Sigma$ of type $(N,k,\ell,m)$ satisfies
\[
\lambda_{1}(\Sigma)=2 ,
\]
and the smallest nonzero eigenvalue of the $\cG$-invariant sector of the Laplacian of $\Sigma$
equals $4+O(m^{-1})$, the constant in the error term depending on $N$, $k$ and $\ell$. In
particular, after enlarging $m_{0}$ if necessary, Yau's conjecture holds for every surface
constructed in \cite{Wiygul} with $m\ge m_{0}$.
\end{theoremA}

The range $N\ge3$ is the substance of Theorem A. For $N=2$ the surfaces are doublings; the Clifford
torus doublings of \cite{KM24} are treated there by the even--odd method, and we make no claim of
priority in the case $N=2$. In particular, to the best of our knowledge, Yau’s conjecture had not previously been verified for any multilayer chain of Clifford tori. Here we prove it for arbitrary finite chain length $N\ge3$.

The threshold $m_{0}$ is neither effective, since it must exceed the non-explicit threshold of \cite[Theorem 6.50]{Wiygul}, nor uniform in $N$. The conformal half-length $a_{i}$ of the tunnels satisfies $a_{i}=\frac{k\ell}{4\pi}\bigl(1-\cos\frac{\pi}{N}\bigr)m^{2}+O_{N,k,\ell}(1)$ as $m\to\infty$ (Proposition \ref{prop:conductance} and Remark \ref{rem:regime}), and the coefficient of $m^{2}$ is asymptotic to $k\ell\pi/(8N^{2})$ as $N\to\infty$; this formally suggests the joint long-neck scale $\sqrt{k\ell}\,m/N\to\infty$, but no two-parameter estimate uniform in $N$ and $m$ is claimed. The limiting graph value $4$ is independent of $N$, and this is exactly what excludes the collapse of the spectral gap described in Section \ref{ss:chains}.

The linearised balancing conditions of the construction determine the numbers $b_{i}$ defined in \cite[(3.14)]{Wiygul}, which are the waist-radius ratios for the unperturbed parameters $\zeta=0$, together with the constant $b_{2}$ which enters the exponent of the waist radii; as $m\to\infty$ they converge to limiting values $b_{i}[N]$ determined by a tridiagonal system (Section \ref{sec:surfaces}). Write $L(P_{N})$ and $A(P_{N})$ for the combinatorial Laplacian and the adjacency operator of the path graph on $N$ vertices, $D$ for an oriented incidence matrix of $P_{N}$, and $\lambda_{1}(L(P_{N}))$ for the spectral gap.

\begin{theoremB}
Let $N\ge2$.
\begin{enumerate}
\item[(i)] The limiting waist ratios are
\[
b_{i}[N]=\frac{\sin(i\pi/N)}{\sin(\pi/N)},\qquad 1\le i\le N-1,
\]
so that the balancing constant is $b_{2}[N]=2\cos(\pi/N)$, the largest eigenvalue of $A(P_{N-1})$, and $\lambda_{1}(L(P_{N}))=2-b_{2}[N]$.
\item[(ii)] For every stacked Clifford torus of type $(N,k,\ell,m)$ the total conductance of the
tunnels joining two adjacent tori is
\[
C_{i}=\frac{8\pi^{2}}{2-b_{2}[N]}\bigl(1+O(m^{-2})\bigr)=:C_{N}\bigl(1+O(m^{-2})\bigr),
\qquad \frac{C_{N}}{2\pi^{2}}\,\lambda_{1}\bigl(L(P_{N})\bigr)=4 ,
\]
and the lowest Rayleigh quotient of the weighted path graph model of the invariant sector
(Definition \ref{def:graph}) equals $4+O(m^{-2})$; the constants in the error terms depend on
$N$, $k$ and $\ell$.
\item[(iii)] As $N\to\infty$, $\lambda_{1}(L(P_{N}))=\pi^{2}N^{-2}+O(N^{-4})$ and
$C_{N}=8N^{2}+O(1)$.
\end{enumerate}
\end{theoremB}

The cancellation in (ii) is exact and uses no closed form: it rests on the incidence identity
$L(P_{N})=DD^{\mathsf T}$, $D^{\mathsf T}D=2I-A(P_{N-1})$, which identifies the nonzero spectrum of
$L(P_{N})$ with $2-\operatorname{spec}A(P_{N-1})$, and on the observation that the vector of limiting
waist ratios is a \emph{positive} eigenvector of $A(P_{N-1})$ with eigenvalue $b_{2}[N]$, hence the
Perron vector. The constant $4$ that survives is the coefficient of the Jacobi operator
\begin{equation}\label{eq:jacobi}
\cJ_{\bT}=\Delta_{\bT}+|A|^{2}+\Ric(\nu,\nu)=\Delta_{\bT}+4
\end{equation}
of the Clifford torus, which enters the construction through the linearisation producing the
balancing conditions. In Section \ref{sec:remarks} we prove that among trees the identity
$\lambda_{1}(L(T))=2-\lambda_{\max}(A(\Line(T)))$ characterises paths (Proposition
\ref{prop:trees}); the mechanism is therefore intrinsically one-dimensional.

\subsection{An abstract reduction}\label{ss:abstractintro}
The proof of Theorem A uses only a short list of properties of the surfaces $\Sigma$, and it is
convenient to isolate them. Suppose a closed surface is cut into \emph{blocks} indexed by the
vertices of a finite connected graph $\mathsf G=(V,E)$ and \emph{channels} indexed by its edges, each
edge $e$ carrying $M_{e}$ channels conformally modelled on a flat cylinder $[-a_{e},a_{e}]\times\bS^{1}$
whose conformal factor decays exponentially away from the ends, where it is of size $r_{e}$; suppose a
finite group $\cG$ of isometries preserves each block, permutes the channels of each edge and acts on
their coordinates by $(t,\theta)\mapsto(t,\pm\theta+\theta_{0})$; and suppose a Poincar\'e inequality
with constant $\mathsf P$ and an aggregate trace inequality with constant $\mathsf K$ hold on each block
for $\cG$-invariant functions. This is Definition \ref{def:decomp}. Let $A_{v}$ be the area of the
$v$-th block, $\Theta_{v}:=\sum_{e\ni v}M_{e}\pi r_{e}^{2}$ the area which the channels have removed
from it, $\widetilde A_{v}:=A_{v}+\Theta_{v}$, and $C_{e}:=M_{e}\pi/a_{e}$ the total conductance of the
$e$-th family; let $\mu(\widetilde A,C)$ be the lowest Rayleigh quotient of the graph $\mathsf G$ with
these vertex and edge weights (Definition \ref{def:graph}), and set
\[
\Xi:=\varepsilon+\eta+\mathsf P+\sum_{e\in E}M_{e}^{1/2}r_{e}^{2}+\sum_{e\in E}M_{e}^{-1/2},
\]
where $\varepsilon$ measures the conformal distortion of the channels and $\eta$ the deviation of the
area of a half-channel from $\pi r_{e}^{2}$.

\begin{theoremC}
There are constants $C^{\ast}>0$ and $\Xi_{0}>0$, depending only on $|V|$, $|E|$, on two-sided
bounds for the $\widetilde A_{v}$, on an upper bound for the $C_{e}$ and on $\mathsf K$, such that
if $\Xi\le\Xi_{0}$ then
\[
\Bigl|\ \min\Bigl\{\cR(u):u\in H^{1}(\Sigma)^{\cG}\setminus\{0\},\ \textstyle\int_{\Sigma}u=0\Bigr\}
-\mu(\widetilde A,C)\ \Bigr|\ \le\ C^{\ast}\,\Xi .
\]
If in addition $\Sigma\subset\bS^{3}$ is a closed embedded minimal surface, $\cG<O(4)$ is generated by
reflections in great spheres, and $\mu(\widetilde A,C)>2+C^{\ast}\Xi$, then $\lambda_{1}(\Sigma)=2$.
\end{theoremC}

Theorem A is the case where $\mathsf G$ is the path $P_{N}$, the blocks are the tori, the channels are
the tunnels, and $\Xi=O(m^{-1})$; the verification is Proposition \ref{prop:graph}. We stress the
limits of this statement. Theorem C is a reduction, not an existence result: it converts the
spectral problem into the computation of $\mu(\widetilde A,C)$, and it says nothing about which
weights occur. Both inputs which make the computation succeed here, namely the exponential decay of
the conformal factor along a catenoidal neck and the balancing conditions which fix the $r_{e}$, must
be supplied by the construction at hand, and verifying the hypotheses for another family requires
work of the kind carried out in Sections \ref{sec:surfaces} and \ref{sec:sector}. Doublings and
stackings of the equatorial two-sphere, and stackings of the equatorial disc in the free boundary
setting, seem natural candidates; we make no claim about them here.

The reduction of chambers joined by thin necks to a weighted discrete graph is classical \cite{JimboMorita,Arrieta,Anne}. What is specific here is that the reduction is carried out in a
symmetry sector, in which the blocks are spectrally rigid and collapse to single vertices; we give a self-contained quantitative version with an explicit defect and hypotheses checkable on gluing constructions.

\subsection{Method}
By Takahashi's theorem $\lambda_{1}(\Sigma)\le2$. If $\lambda_{1}(\Sigma)<2$, the reflection lemma
of Choe and Soret \cite{ChoeSoret}, in a formulation which uses only minimality, Takahashi's
theorem and Courant's nodal domain theorem (Lemma \ref{lem:CS}), shows that every first
eigenfunction is invariant under the group $\cG$, which is generated by reflections in great
spheres. It therefore suffices to bound the
Rayleigh quotient from below on the $\cG$-invariant sector. On each torus the $\cG$-invariant
functions have Fourier support in a lattice of frequencies of order $m$, so a $\cG$-invariant
function is, up to an error controlled by its Dirichlet energy, constant on each torus; the
Dirichlet energy of such a function is bounded from below by the energy of the interpolating mode in
the tunnels, and the latter is controlled by a sharp channel inequality. The resulting lower bound
is the lowest Rayleigh quotient $\mu$ of a weighted path graph whose weights are the areas of the
tori and the conductances of the tunnel families, and Theorem B evaluates $\mu=4+O(m^{-2})$. Two
elementary inequalities carry the analysis: the channel inequality on a cylinder (Lemma
\ref{lem:channel}), which is conformally invariant and imposes no condition on the boundary data,
so that it applies to the restriction of an arbitrary function to a tunnel; and a Poincar\'e
inequality on the periodically perforated torus (Lemma \ref{lem:poincare}), available with a
constant independent of $m$ because the ratio of the gluing radius to the lattice spacing is fixed
in \cite{Wiygul}. We use neither capacity theory nor spectral convergence on graph-like spaces, and
we work on the minimal surface itself.

\subsection{What is used from the construction}
Numbered references to \cite{Wiygul} follow the published version. From that paper we use: the definition of the initial surfaces \cite[(2.1)--(2.9), (2.26)--(2.30)]{Wiygul}; the
definition of the waist radii and the balancing conditions \cite[(2.12)--(2.18), (3.11)--(3.17)]{Wiygul}
together with the existence statement \cite[Lemma 3.18]{Wiygul} and the bounds
\cite[(3.6)]{Wiygul} on the heights; the parametrisation \cite[(4.22)--(4.25)]{Wiygul} of the catenoidal regions; the statement \cite[(4.4)]{Wiygul} that $\cG$ acts on normal perturbations
without sign; the conformal factor $\rho$ of \cite[(4.9)--(4.15), (4.27)]{Wiygul}, the weighted H\"older norms of \cite[(4.6)]{Wiygul} and the decay norms of the paragraph ``Decay norms and a global estimate of the mean curvature'' of \cite[\S4]{Wiygul}, whose exponents $\alpha,\gamma\in(0,1)$
are fixed once and for all in the hypotheses of \cite[Theorem 6.50]{Wiygul}; and the main theorem \cite[Theorem 1.1, Lemma 4.42, Theorem 6.50, (6.61)]{Wiygul}, which produces the minimal surface as the normal graph over the
initial surface of a function $u$ satisfying $\|u\|_{2,\alpha,\gamma}\le2C_{1}\tau_{1}$. In Section
\ref{ss:minimal} we show that this estimate implies property (W) of Definition \ref{def:stacked},
the only property of the final surface used afterwards; the derivation uses of the weight only the
positivity of its exponent $\gamma$, and of $\rho$ only the two-sided bound $m\le\rho\le C\tau_{1}^{-1}$.
All estimates on the initial surface which we need are derived in Section \ref{sec:surfaces} from
the explicit formulas.

\subsection{Organisation}
Section \ref{sec:prelim} collects Takahashi's theorem and the reflection lemma. Section \ref{sec:surfaces} records the data of the construction, defines the class of
surfaces and shows that the surfaces of \cite{Wiygul} belong to it. Section \ref{sec:balancing}
proves Theorem B(i). Section \ref{sec:channel} proves the channel inequality and evaluates the
conductances. Section \ref{sec:sector} establishes the Poincar\'e, trace and mass estimates, proves
the abstract reduction theorem (Theorem C, restated as Theorem \ref{thm:abstract}) and applies it to
the invariant sector. Section \ref{sec:proof} proves Theorems A
and B. Section \ref{sec:remarks} discusses the structure of the identity, analyses the balancing
operator of a tree, characterises paths among trees, and states two problems.

\begin{convention}\label{conv:notation}
All surfaces are smooth, closed, embedded and connected. The Laplacian is
$\Delta=\operatorname{div}\nabla$, so that $\Delta\varphi+\lambda\varphi=0$ for an eigenfunction with
eigenvalue $\lambda\ge0$. We write $|\Omega|$ for the area of a two-dimensional region, and
\[
\cR(u)=\frac{\int_{\Sigma}|\nabla u|^{2}}{\int_{\Sigma}u^{2}},\qquad
\lambda_{1}(\Sigma)=\min\Bigl\{\cR(u):u\in H^{1}(\Sigma)\setminus\{0\},\
\textstyle\int_{\Sigma}u=0\Bigr\}.
\]
Unless stated otherwise, constants denoted $C$, $C'$, $c$, $K$, \dots\ depend only on $N$, $k$ and
$\ell$ and may change from line to line; $O(m^{-1})$ and $O(m^{-2})$ are understood with such
constants, uniformly for $m\ge m_{0}[N,k,\ell]$, and $m_{0}$ may be enlarged finitely many times.
No uniformity in $N$ is claimed for the constants or for $m_{0}$; see Remark \ref{rem:regime}.
\end{convention}

\section{Preliminaries}\label{sec:prelim}

\begin{lemma}[Takahashi \cite{Takahashi}]\label{lem:takahashi}
Let $\Sigma\subset\bS^{3}$ be a closed minimal surface. For every linear function $\xi$ on $\bR^{4}$
the restriction $\xi|_{\Sigma}$ satisfies $\Delta_{\Sigma}\xi+2\xi=0$ and $\int_{\Sigma}\xi=0$.
Consequently $\lambda_{1}(\Sigma)\le2$.
\end{lemma}

\begin{proof}
The equation is \eqref{eq:takahashi} with $n=2$; integrating it over the closed surface gives
$\int_{\Sigma}\xi=0$. Since $\xi|_{\Sigma}\not\equiv0$ for a suitable $\xi$, the variational
characterisation gives $\lambda_{1}\le2$.
\end{proof}

The following lemma is the first half of the argument of Choe and Soret \cite{ChoeSoret}. We
include a proof because its hypotheses matter: it uses minimality only through Takahashi's theorem,
together with Courant's nodal domain theorem, and nothing else; in particular it is independent of
any combinatorial hypothesis on a fundamental domain.

\begin{lemma}[Reflection lemma]\label{lem:CS}
Let $\Sigma\subset\bS^{3}$ be a closed embedded minimal surface invariant under a finite group
$\cG<O(4)$ generated by reflections in great spheres. If $\lambda_{1}(\Sigma)<2$, then every first
eigenfunction of $\Sigma$ is $\cG$-invariant.
\end{lemma}

\begin{proof}
Assume $\lambda_{1}:=\lambda_{1}(\Sigma)<2$, let $u$ be a real first eigenfunction, let
$\sigma\in\cG$ be the reflection in a great sphere $\Pi$, and set $\psi:=u-u\circ\sigma$. Since
$\sigma$ restricts to an isometry of $\Sigma$, $u\circ\sigma$ is a $\lambda_{1}$-eigenfunction, so
either $\psi\equiv0$ or $\psi$ is a $\lambda_{1}$-eigenfunction; we assume the latter and derive a
contradiction. Note that $\psi\circ\sigma=-\psi$ and that $\psi=0$ on $\Sigma\cap\Pi$, because
$\sigma$ fixes $\Pi$ pointwise.

By Courant's nodal domain theorem \cite[Ch.~VI, \S6]{CourantHilbert} a $\lambda_{1}$-eigenfunction
has at most two nodal domains. Since $\int_{\Sigma}\psi=0$ and $\psi\not\equiv0$, the open sets
$\Omega_{+}:=\{\psi>0\}$ and $\Omega_{-}:=\{\psi<0\}$ are both nonempty; hence each of them is a
single nodal domain, and in particular each of them is connected. Since $\psi\circ\sigma=-\psi$,
the reflection $\sigma$ maps $\Omega_{+}$ onto $\Omega_{-}$.

Let $\xi$ be a linear function on $\bR^{4}$ vanishing exactly on the hyperplane spanned by $\Pi$,
so that $\xi\circ\sigma=-\xi$, and let $D_{\pm}:=\Sigma\cap\{\pm\xi>0\}$. These are disjoint open
subsets of $\Sigma$ with $D_{+}\cup D_{-}=\Sigma\setminus\Pi$, and $\sigma(D_{+})=D_{-}$. Since
$\psi=0$ on $\Sigma\cap\Pi$, the connected set $\Omega_{+}$ is contained in
$\Sigma\setminus\Pi=D_{+}\sqcup D_{-}$, hence in one of the two sets $D_{\pm}$; replacing $u$ by
$-u$, which replaces $\psi$ by $-\psi$ and interchanges $\Omega_{+}$ and $\Omega_{-}$, we may
assume $\Omega_{+}\subset D_{+}$. Then $\Omega_{-}=\sigma(\Omega_{+})\subset\sigma(D_{+})=D_{-}$.
Consequently $\psi\ge0$ on $D_{+}$, because $\{\psi<0\}=\Omega_{-}\subset D_{-}$; likewise
$\psi\le0$ on $D_{-}$; and $\psi=0$ on $\Sigma\cap\Pi$. Since $\xi>0$ on $D_{+}$, $\xi<0$ on
$D_{-}$ and $\xi=0$ on $\Pi$, we conclude that $\psi\,\xi\ge0$ on $\Sigma$ and $\psi\,\xi>0$ on the
nonempty open set $\Omega_{+}$; hence $\int_{\Sigma}\psi\,\xi>0$.

On the other hand, by Lemma \ref{lem:takahashi}, $\Delta\xi+2\xi=0$ on $\Sigma$, while
$\Delta\psi+\lambda_{1}\psi=0$, so that Green's formula on the closed surface $\Sigma$ gives
\[
(2-\lambda_{1})\int_{\Sigma}\psi\,\xi=\int_{\Sigma}\bigl(\xi\,\Delta\psi-\psi\,\Delta\xi\bigr)=0 ,
\]
and $\int_{\Sigma}\psi\,\xi=0$ because $\lambda_{1}\ne2$: a contradiction. Therefore $\psi\equiv0$,
that is $u\circ\sigma=u$ for every reflection $\sigma\in\cG$ in a great sphere, and $u$ is
$\cG$-invariant since $\cG$ is generated by such reflections.
\end{proof}

\begin{remark}\label{rem:CS}
In \cite{ChoeSoret} the nodal domains of $\psi$ are identified with the two components of
$\Sigma\setminus\Pi$ furnished by the two-piece property of Ros \cite{Ros}, and the contradiction
is derived on one of them. The orthogonality argument above makes this identification unnecessary:
it uses neither the two-piece property nor embeddedness, and it applies verbatim to a closed
minimal surface immersed in $\bS^{3}$ on which a reflection of $\bS^{3}$ induces an isometry. We
shall use the lemma only for embedded surfaces, which is the setting of Yau's conjecture.
\end{remark}

\section{The surfaces of Wiygul}\label{sec:surfaces}

In this section all numbered references are to the published version of \cite{Wiygul}.

\subsection{The Clifford torus and flat coordinates}
Realise $\bS^{3}=\{(z_{1},z_{2})\in\bC^{2}:|z_{1}|^{2}+|z_{2}|^{2}=1\}$, let
$\bT=\{|z_{1}|=|z_{2}|=1/\sqrt2\}$ be the Clifford torus, and let $C_{1}=\{z_{2}=0\}$ and
$C_{2}=\{z_{1}=0\}$ be its axes. The map \cite[(2.2)]{Wiygul}
\begin{equation}\label{eq:Phi}
\Phi(\rx,\ry,\rz):=\Bigl(e^{i\sqrt2\rx}\sin\bigl(\rz+\tfrac\pi4\bigr),\
e^{i\sqrt2\ry}\cos\bigl(\rz+\tfrac\pi4\bigr)\Bigr),\qquad
(\rx,\ry,\rz)\in\bR\times\bR\times\bigl(-\tfrac\pi4,\tfrac\pi4\bigr),
\end{equation}
is a covering of $\bS^{3}\setminus(C_{1}\cup C_{2})$, and \cite[(2.3)]{Wiygul}
\begin{equation}\label{eq:Phimetric}
\Phi^{*}g_{\bS^{3}}=d\rx^{2}+d\ry^{2}+d\rz^{2}+\sin(2\rz)\bigl(d\rx^{2}-d\ry^{2}\bigr).
\end{equation}
The horizontal planes $\{\rz=\text{const}\}$ are mapped onto the tori $\bT_{\rz}$ parallel to
$\bT=\bT_{0}$, on which the induced metric is $(1+\sin2\rz)d\rx^{2}+(1-\sin2\rz)d\ry^{2}$ with area
element $\cos(2\rz)\,d\rx\,d\ry$. In particular $\Phi(\cdot,\cdot,0)$ descends to an isometry of the
flat torus
\[
\Tfl:=\bR^{2}/(\sqrt2\pi\bZ)^{2}
\]
onto $\bT$, so that
\begin{equation}\label{eq:areaT}
|\bT|=2\pi^{2},
\end{equation}
and the eigenvalues of the Laplacian of $\bT$ are $2(p^{2}+q^{2})$, $p,q\in\bZ$, with
eigenfunctions $e^{i\sqrt2(p\rx+q\ry)}$. The principal curvatures of $\bT$ are $\pm1$, so
$|A|^{2}=2$, and $\Ric=2g$ on $\bS^{3}$; the Jacobi operator of $\bT$ is $\cJ_{\bT}=\Delta_{\bT}+4$
as in \eqref{eq:jacobi}. A \emph{parallel torus} is one of the tori $\bT_{\rz}$.

\subsection{The symmetry group}
Fix integers $k,\ell,m\ge1$ and set
\begin{equation}\label{eq:XY}
X:=\frac{\pi}{\sqrt2\,km},\qquad Y:=\frac{\pi}{\sqrt2\,\ell m},\qquad R:=\frac{1}{10\,\ell m},
\qquad M:=k\ell m^{2}.
\end{equation}
Following \cite[(2.10)]{Wiygul} we assume $k\le\ell$; this is no loss of generality, since
interchanging the axes $C_{1},C_{2}$ interchanges $k$ and $\ell$. Let
$\underline{\mathsf X}(z_{1},z_{2})=(\overline{z_{1}},z_{2})$ and
$\underline{\mathsf Y}(z_{1},z_{2})=(z_{1},\overline{z_{2}})$, and let
$\mathsf R_{C_{2}}^{\theta}(z_{1},z_{2})=(e^{i\theta}z_{1},z_{2})$,
$\mathsf R_{C_{1}}^{\theta}(z_{1},z_{2})=(z_{1},e^{i\theta}z_{2})$. The symmetry group of the
construction is \cite[(2.26)]{Wiygul}
\begin{equation}\label{eq:G}
\cG=\cG[k,\ell,m]:=\bigl\langle\mathsf R_{C_{2}}^{2\pi/km},\ \mathsf R_{C_{1}}^{2\pi/\ell m},\
\underline{\mathsf X},\ \underline{\mathsf Y}\bigr\rangle\cong D_{km}\times D_{\ell m},
\end{equation}
where $D_{q}$ is the dihedral group of order $2q$. Through $\Phi$ \cite[(2.25)]{Wiygul}, $\cG$ acts
on the coordinates $(\rx,\ry,\rz)$ by the group generated by the translations
$(\rx,\ry)\mapsto(\rx+2X,\ry)$, $(\rx,\ry)\mapsto(\rx,\ry+2Y)$ and the reflections
$\rx\mapsto-\rx$, $\ry\mapsto-\ry$, the coordinate $\rz$ being fixed. In particular every element of
$\cG$ preserves each parallel torus and each of its two sides.

\begin{lemma}\label{lem:reflgen}
$\cG[k,\ell,m]$ is generated by reflections in great spheres of $\bS^{3}$.
\end{lemma}

\begin{proof}
Each of $\underline{\mathsf X}$, $\underline{\mathsf Y}$ fixes pointwise a three-dimensional linear
subspace of $\bR^{4}$ and reverses its orthogonal complement, hence is the reflection in a great
sphere; so is every conjugate $\mathsf R\,\underline{\mathsf X}\,\mathsf R^{-1}$ by a rotation
$\mathsf R$ about $C_{1}$ or $C_{2}$. For $\theta\in\bR$ one computes
$\mathsf R_{C_{2}}^{\theta}\underline{\mathsf X}\mathsf R_{C_{2}}^{-\theta}(z_{1},z_{2})
=(e^{2i\theta}\overline{z_{1}},z_{2})$, that is
$\mathsf R_{C_{2}}^{\theta}\underline{\mathsf X}\mathsf R_{C_{2}}^{-\theta}
=\mathsf R_{C_{2}}^{2\theta}\underline{\mathsf X}$. Hence the reflection
$\mathsf R_{C_{2}}^{\pi/km}\underline{\mathsf X}\mathsf R_{C_{2}}^{-\pi/km}
=\mathsf R_{C_{2}}^{2\pi/km}\underline{\mathsf X}$ belongs to $\cG$, although $\mathsf R_{C_{2}}^{\pi/km}$
itself does not, and together with $\underline{\mathsf X}$ it generates the dihedral factor $D_{km}$,
the product $\underline{\mathsf X}\cdot\mathsf R_{C_{2}}^{2\pi/km}\underline{\mathsf X}$ being
$\mathsf R_{C_{2}}^{-2\pi/km}$. Similarly for $D_{\ell m}$.
\end{proof}

\subsection{Lattices, cells and the gluing radius}
Let $\mathsf L_{0}\subset\Tfl$ be the image of $2X\bZ\times2Y\bZ$ and $\mathsf L_{1}:=\mathsf L_{0}+(X,Y)$
\cite[(2.27)]{Wiygul}; both are $\cG$-invariant rectangular lattices of cell size $2X\times2Y$, each
with $M=k\ell m^{2}$ points, and $\mathsf L_{0}\cup\mathsf L_{1}$ is a lattice with $2M$ points. The
tunnels of the $i$-th layer, joining the $i$-th and $(i+1)$-st tori, sit over
$\mathsf L_{(i+1)\bmod2}$ \cite[\S1, (2.29)]{Wiygul}. Accordingly the set of \emph{removal points}
of the $j$-th torus is
\begin{equation}\label{eq:Zj}
Z_{1}:=\mathsf L_{0},\qquad Z_{N}:=\mathsf L_{N\bmod2},\qquad
Z_{j}:=\mathsf L_{0}\cup\mathsf L_{1}\quad(1<j<N),
\end{equation}
with $M_{j}:=\#Z_{j}$ equal to $M$ for $j\in\{1,N\}$ and to $2M$ otherwise, and the
\emph{perforated tori} are
\begin{equation}\label{eq:Pj}
P_{j}:=\Tfl\setminus\bigcup_{p\in Z_{j}}B(p,R).
\end{equation}
Since $k\le\ell$, any two distinct points of $\mathsf L_{0}\cup\mathsf L_{1}$ are at distance at least
$\min\{2Y,\sqrt{X^{2}+Y^{2}}\}\ge Y$, and
\begin{equation}\label{eq:ratio}
Y=\frac{\pi}{\sqrt2\,\ell m}>\frac{2.2}{\ell m}>4R,\qquad \frac{R}{2Y}=\frac{1}{10\sqrt2\,\pi}<\frac1{44}.
\end{equation}
The ratio of the gluing radius $R$ to the lattice spacing is thus a constant independent of $m$; this
is what makes the constants of Section \ref{sec:sector} independent of $m$.

\subsection{Waist radii and balancing conditions}
Let $n:=\lfloor N/2\rfloor$. For $N=2$ and $N=3$ set $b_{2}:=0$ and $b_{2}:=1$ respectively
\cite[(2.12)]{Wiygul}. For $N\ge4$ the numbers $b_{2},\dots,b_{n}$ are a solution of the tridiagonal
system \cite[(3.17)]{Wiygul}
\begin{equation}\label{eq:system}
\begin{aligned}
-b_{i-1}+b_{2}b_{i}-b_{i+1}&=-\frac{8\pi}{k\ell m^{2}}\,b_{i}\ln b_{i}, && 2\le i\le n-1,\\
-2^{(N+1)\bmod2}\,b_{n-1}+\bigl(b_{2}-(N\bmod2)\bigr)b_{n}&=-\frac{8\pi}{k\ell m^{2}}\,b_{n}\ln b_{n},
\end{aligned}
\end{equation}
with $b_{1}:=1$; in all cases we extend the family by $b_{i}:=b_{N-i}$ for $n<i\le N-1$, so that
$b_{1}=b_{N-1}=1$. We call \emph{limiting system} the system \eqref{eq:system} with the right-hand
sides replaced by $0$. From \cite[Lemma 3.18 and its proof]{Wiygul} we use the following facts.
\begin{enumerate}
\item[(B1)] For $N\ge4$ the limiting system has a solution $d=(d_{2}[N],\dots,d_{n}[N])$ with $d_{2}\in(1,2)$ and, when $n\ge3$, $d_{2}<d_{3}<\dots<d_{n}$. For $N\ge6$ the differential at $d$ of the map $F_{N}:x\mapsto A_{N}(x_{2})x$ of \cite[(3.28)]{Wiygul}, in terms of which the limiting system reads $F_{N}(x)=e_{1}$ \cite[(3.21)]{Wiygul}, is invertible \cite[(3.29)--(3.31)]{Wiygul}.
\item[(B2)] For $m\ge m_{0}[N]$ the numbers $b_{i}=b_{i}[N,k,\ell,m]$ are the solution of
\eqref{eq:system} near $d$ furnished by the inverse function theorem, and $b_{i}\to d_{i}[N]$ as
$m\to\infty$.
\end{enumerate}
The waist radii are defined by \cite[(2.13)--(2.14)]{Wiygul}: with parameters
$\zeta=(\zeta_{1},\dots,\zeta_{N-1})$ bounded by a constant $c[N,k,\ell]$ independent of $m$,
\begin{equation}\label{eq:tau}
\tau_{1}=\frac{e^{\zeta_{1}}}{10\ell m}\exp\Bigl(-\frac{k\ell m^{2}}{4\pi}\Bigl(1-\frac{b_{2}}{2}\Bigr)\Bigr),
\qquad
\tau_{i}=e^{\zeta_{i}/(k\ell m^{2})}\,b_{i}\,\tau_{1}\quad(1<i<N),
\end{equation}
so that \cite[(2.18)]{Wiygul}
\begin{equation}\label{eq:lntau}
\ln\frac{1}{10\ell m\tau_{i}}=\frac{k\ell m^{2}}{4\pi}\Bigl(1-\frac{b_{2}}{2}\Bigr)-\zeta_{1}
-(1-\delta_{i1})\Bigl(\ln b_{i}+\frac{\zeta_{i}}{k\ell m^{2}}\Bigr).
\end{equation}
The factors $4\pi$ and $8\pi$ arise from the linearisation of the minimal surface equation about
$\bT$, that is from the Jacobi operator \eqref{eq:jacobi}: in the force computation
\cite[(3.5)--(3.7)]{Wiygul} the mean curvature $2\tan2\rz\approx4\rz$ of the parallel torus $\bT_{\rz}$
is integrated over a cell of area $2\pi^{2}/(k\ell m^{2})$ and balanced against the flux $2\pi\tau$
of a catenoidal neck, and $8\pi^{2}=4\cdot|\bT|$. The heights of the tori and of the necks satisfy
\cite[(3.6)]{Wiygul}
\begin{equation}\label{eq:heights}
\max_{j}|\rz_{j}|+\max_{i}|\rz^{K}_{i}|\le C\,m^{2}\tau_{1}.
\end{equation}
In particular, since $b_{2}\le2-1/C$ by \cite[(2.12)]{Wiygul}, $\tau_{1}$ decays like
$\exp(-cm^{2})$ with $c=c[N,k,\ell]>0$, and $m^{q}\tau_{1}\to0$ for every fixed $q$; moreover, since
$b_{i}\to d_{i}[N]>0$ by (B2) and $|\zeta|\le c$, all the ratios $\tau_{i}/\tau_{1}$ lie in a compact
subset of $(0,\infty)$ depending only on $N$, for $m\ge m_{0}$.

\subsection{Catenoidal regions of the initial surface}
Let $\bK_{a}:=[-a,a]\times\bS^{1}$ with coordinates $(t,\theta)$, $\theta\in\bR/2\pi\bZ$, and flat
metric $\widehat\chi:=dt^{2}+d\theta^{2}$. For $1\le i\le N-1$ let \cite[(4.24)]{Wiygul}
\begin{equation}\label{eq:ai}
a_{i}:=\arcosh\frac{1}{10\ell m\tau_{i}},\qquad\text{so that}\qquad \tau_{i}\cosh a_{i}=R ,
\end{equation}
and, by \eqref{eq:lntau} and $\arcosh x=\ln x+\ln(1+\sqrt{1-x^{-2}})$,
\begin{equation}\label{eq:aiexp}
a_{i}=\frac{k\ell m^{2}}{4\pi}\Bigl(1-\frac{b_{2}}{2}\Bigr)-\zeta_{1}
-(1-\delta_{i1})\Bigl(\ln b_{i}+\frac{\zeta_{i}}{k\ell m^{2}}\Bigr)
+\ln\bigl(1+\sqrt{1-100\ell^{2}m^{2}\tau_{i}^{2}}\bigr).
\end{equation}
The $i$-th catenoidal region of the initial surface is the image of the map \cite[(4.23), (4.25)]{Wiygul}
\begin{equation}\label{eq:kappa}
\kappa_{i}:\bK_{a_{i}}\to\bS^{3},\qquad
\kappa_{i}(t,\theta):=\Phi\bigl(\rx_{i}+\tau_{i}\cosh t\cos\theta,\ \ry_{i}+\tau_{i}\cosh t\sin\theta,\
\rz^{K}_{i}+\tau_{i}t\bigr),
\end{equation}
where $(\rx_{i},\ry_{i})=((i-1)X,(i-1)Y)$ is a point of $\mathsf L_{(i+1)\bmod2}$; the other tunnels
of the $i$-th layer are the images of $\kappa_{i}$ under $\cG$, and they are pairwise disjoint. The
end circles of a tunnel are the images of $\{t=\pm a_{i}\}$; their projections to the
$(\rx,\ry)$-plane are the circles of radius $R$ about the removal point, and the coordinate $\theta$
is the polar angle about that point.

\begin{lemma}\label{lem:initialcat}
Let $g^{0}$ be the metric induced on the initial surface. Then
\begin{gather*}
\kappa_{i}^{*}g^{0}=\tau_{i}^{2}\cosh^{2}t\,\Bigl(\widehat\chi+\sin\bigl(2\rz^{K}_{i}+2\tau_{i}t\bigr)\,\mathsf B\Bigr),\\
\mathsf B:=\tanh^{2}t\cos2\theta\,dt^{2}-2\tanh t\sin2\theta\,dt\,d\theta-\cos2\theta\,d\theta^{2},
\end{gather*}
and consequently, for $m\ge m_{0}$,
\begin{equation}\label{eq:catcomp}
(1-Cm^{2}\tau_{1})\,\tau_{i}^{2}\cosh^{2}t\,\widehat\chi\ \le\ \kappa_{i}^{*}g^{0}\ \le\
(1+Cm^{2}\tau_{1})\,\tau_{i}^{2}\cosh^{2}t\,\widehat\chi
\qquad\text{on }\bK_{a_{i}}.
\end{equation}
\end{lemma}

\begin{proof}
Write $\kappa_{i}=\Phi\circ\widehat\kappa_{i}$ with $\widehat\kappa_{i}(t,\theta)=(\rx_{i},\ry_{i},\rz^{K}_{i})
+\tau_{i}(\cosh t\cos\theta,\cosh t\sin\theta,t)$. Then
$d\rx=\tau_{i}(\sinh t\cos\theta\,dt-\cosh t\sin\theta\,d\theta)$,
$d\ry=\tau_{i}(\sinh t\sin\theta\,dt+\cosh t\cos\theta\,d\theta)$, $d\rz=\tau_{i}dt$, so that
$d\rx^{2}+d\ry^{2}+d\rz^{2}=\tau_{i}^{2}\cosh^{2}t\,\widehat\chi$ and
$d\rx^{2}-d\ry^{2}=\tau_{i}^{2}\cosh^{2}t\,\mathsf B$; the formula follows from
\eqref{eq:Phimetric} with $\rz=\rz^{K}_{i}+\tau_{i}t$. The symmetric matrix of $\mathsf B$ with
respect to $\widehat\chi$ has entries bounded by $1$ in absolute value, so
$|\mathsf B(v,v)|\le2\widehat\chi(v,v)$. Finally $|\rz^{K}_{i}+\tau_{i}t|\le|\rz^{K}_{i}|+\tau_{i}a_{i}
\le Cm^{2}\tau_{1}$ by \eqref{eq:heights} and \eqref{eq:aiexp}, and $|\sin s|\le|s|$.
\end{proof}

\subsection{Toral regions of the initial surface}
The initial surface $\Sigma^{0}=\Sigma^{0}[N,k,\ell,m,\zeta,\xi]$ is defined in
\cite[(2.29)--(2.30)]{Wiygul} as the $\cG$-orbit of $N$ pieces $\Omega_{1},\dots,\Omega_{N}$, the
$j$-th of which is the image under $\Phi$ of the graph, over a cell of the plane $\{\rz=\rz_{j}\}$
with one or two discs removed, of the function $\phi$ of \cite[(2.6), (2.9)]{Wiygul}. Explicitly,
in polar coordinates $r$ about a removal point $p$ of the $j$-th torus, with $\tau$ the waist radius
and $\rz^{K}$ the height of the neck at $p$,
\begin{equation}\label{eq:phi}
\phi=\rz^{K}+(\rz_{j}-\rz^{K})\,\psi[R,2R](r)+\operatorname{sgn}(\rz_{j}-\rz^{K})\,
\tau\arcosh\frac{r}{\tau}\,\psi[2R,R](r)\qquad(\tau\le r\le2R),
\end{equation}
and $\phi=\rz_{j}$ at distance at least $2R$ from all removal points, where $\psi[a,b]$ is a fixed
smooth monotone cut-off equal to $0$ at $a$ and to $1$ at $b$ \cite[(2.4)]{Wiygul}; for $r\le R$ the
function $\phi$ is the exact catenoid, and this part is the catenoidal region of the previous
subsection. We define the \emph{initial toral region} $\cT^{0}[j]$ to be the part of the $j$-th torus
of $\Sigma^{0}$ lying over $P_{j}$, that is the graph of $\phi_{j}:=\phi|_{P_{j}}$.

\begin{lemma}\label{lem:initialtor}
For $m\ge m_{0}$,
\[
\sup_{P_{j}}|\phi_{j}-\rz_{j}|\le Cm^{2}\tau_{1},\qquad
\sup_{P_{j}}|\nabla\phi_{j}|\le Cm^{3}\tau_{1},\qquad
\sup_{P_{j}}|\nabla^{2}\phi_{j}|\le Cm^{4}\tau_{1},
\]
and consequently, under the projection $\pi_{j}:\cT^{0}[j]\to P_{j}$, $(\rx,\ry,\phi_{j})\mapsto(\rx,\ry)$,
\begin{equation}\label{eq:torcomp}
(1-Cm^{2}\tau_{1})\,(d\rx^{2}+d\ry^{2})\ \le\ (\pi_{j}^{-1})^{*}g^{0}\ \le\ (1+Cm^{2}\tau_{1})\,(d\rx^{2}+d\ry^{2}).
\end{equation}
The map $\pi_{j}$ is $\cG$-equivariant for the actions of $\cG$ on $\Sigma^{0}$ and on $\Tfl$, and it
maps each end circle of a tunnel adjoining $\cT^{0}[j]$ onto the circle $\partial B(p,R)$ about the
corresponding removal point, the tunnel coordinate $\theta$ corresponding to the polar angle.
\end{lemma}

\begin{proof}
On the annulus $R\le r\le2R$ about a removal point, \eqref{eq:phi} gives
$|\phi-\rz_{j}|\le|\rz_{j}-\rz^{K}|+\tau\arcosh(2R/\tau)$. By \cite[(2.17)]{Wiygul} and
\eqref{eq:heights}, $|\rz_{j}-\rz^{K}|\le Cm^{2}\tau_{1}$, and $\tau\arcosh(2R/\tau)\le
\tau\ln(4R/\tau)\le Cm^{2}\tau_{1}$ by \eqref{eq:lntau}. Put $f(r):=\tau\arcosh(r/\tau)$; for
$r\ge R\ge2\tau$ one has $|f|\le Cm^{2}\tau_{1}$, $|f'|=\tau/\sqrt{r^{2}-\tau^{2}}\le2\tau/R$ and
$|f''|=\tau r/(r^{2}-\tau^{2})^{3/2}\le2\tau/R^{2}$, while the cut-offs satisfy $|\psi'|\le C/R$ and
$|\psi''|\le C/R^{2}$. Since $R^{-1}=10\ell m$ and $\tau\le C\tau_{1}$, the first derivatives of
\eqref{eq:phi} are bounded by $Cm^{2}\tau_{1}\cdot Cm+2\tau/R\le Cm^{3}\tau_{1}$, and the second
derivatives, the Hessian of a radial function $h(r)$ being bounded by $|h''|+|h'|/r$, by
$Cm^{2}\tau_{1}/R^{2}+C\tau/R^{2}\le Cm^{4}\tau_{1}$. Elsewhere on $P_{j}$, $\phi_{j}=\rz_{j}$. The
induced metric of the graph is, by \eqref{eq:Phimetric},
\[
(1+\sin2\phi_{j})\,d\rx^{2}+(1-\sin2\phi_{j})\,d\ry^{2}+(\partial_{\rx}\phi_{j}\,d\rx+\partial_{\ry}\phi_{j}\,d\ry)^{2},
\]
which differs from $d\rx^{2}+d\ry^{2}$ by a form bounded by $(2|\phi_{j}|+|\nabla\phi_{j}|^{2})(d\rx^{2}+d\ry^{2})
\le Cm^{2}\tau_{1}(d\rx^{2}+d\ry^{2})$. Equivariance holds because $\Sigma^{0}$ is the
$\cG$-orbit of $\bigcup_{j}\Omega_{j}$ and $\cG$ acts through $\Phi$ by the affine isometries listed
after \eqref{eq:G}. The statement about the end circles follows from \eqref{eq:kappa} at
$t=\pm a_{i}$, where $\tau_{i}\cosh a_{i}=R$.
\end{proof}

\begin{lemma}[Second fundamental form of the initial surface]\label{lem:initialA}
Let $A^{0}$ be the second fundamental form of $\Sigma^{0}$ and $|A^{0}|$ its norm with respect to
$g^{0}$. For $m\ge m_{0}$,
\[
|A^{0}|\le C\ \text{ on }\cT^{0}[j]\quad(1\le j\le N),\qquad
|A^{0}|\circ\kappa_{i}\le C\,\tau_{i}^{-1}\sech t\ \text{ on }\bK_{a_{i}}\quad(1\le i\le N-1).
\]
\end{lemma}

\begin{proof}
Let $U:=\bR^{2}\times(-\pi/8,\pi/8)$ and $h:=\Phi^{*}g_{\bS^{3}}|_{U}=g_{E}+\sin(2\rz)(d\rx^{2}-d\ry^{2})$,
where $g_{E}$ is the Euclidean metric. Since $|\sin2\rz|\le\sin(\pi/4)<3/4$ on $U$, we have
$\tfrac14g_{E}\le h\le2g_{E}$, and the Christoffel symbols of $h$ in the coordinates $(\rx,\ry,\rz)$,
namely $\Gamma^{\rz}_{\rx\rx}=-\Gamma^{\rz}_{\ry\ry}=-\cos2\rz$,
$\Gamma^{\rx}_{\rx\rz}=\Gamma^{\rx}_{\rz\rx}=\cos2\rz/(1+\sin2\rz)$,
$\Gamma^{\ry}_{\ry\rz}=\Gamma^{\ry}_{\rz\ry}=-\cos2\rz/(1-\sin2\rz)$ and zero otherwise, are bounded
by $4$. By \eqref{eq:heights}, $\Sigma^{0}$ lies in $\Phi(U)$ for $m\ge m_{0}$, and $\Phi$ is a
local isometry of $(U,h)$ onto its image.

Let $F:S\to U$ be an immersion of a surface, with unit normals $\nu_{E}$ and $\nu_{h}$ and second
fundamental forms $A_{E}$ and $A_{h}$ relative to $g_{E}$ and to $h$. For tangent vectors $V,W$,
$A_{h}(V,W)=h\bigl(\nu_{h},\partial_{V}\partial_{W}F+\Gamma(\partial_{V}F,\partial_{W}F)\bigr)$, where
$\Gamma(\cdot,\cdot)$ denotes the bilinear expression in the Christoffel symbols of $h$. By the Gauss
formula in Euclidean space, $\partial_{V}\partial_{W}F=dF(\nabla_{V}W)+A_{E}(V,W)\nu_{E}$ with
$\nabla$ the Levi-Civita connection of $F^{*}g_{E}$, and $\nu_{h}$ is $h$-orthogonal to $dF(TS)$;
hence
\[
A_{h}(V,W)=A_{E}(V,W)\,h(\nu_{h},\nu_{E})+h\bigl(\nu_{h},\Gamma(\partial_{V}F,\partial_{W}F)\bigr).
\]
Here $|h(\nu_{h},\nu_{E})|\le|\nu_{E}|_{h}\le\sqrt2$ and
$|h(\nu_{h},\Gamma(X,Y))|\le\sqrt2\,|\Gamma(X,Y)|_{E}\le C|X|_{E}|Y|_{E}$, which is at most
$C|X|_{h}|Y|_{h}$. Since $F^{*}g_{E}$ and $F^{*}h$ are comparable, $|A_{h}|_{F^{*}h}\le C\bigl(|A_{E}|_{F^{*}g_{E}}+1\bigr)$
with an absolute constant $C$.

On $\cT^{0}[j]$ take $F(\rx,\ry)=(\rx,\ry,\phi_{j}(\rx,\ry))$: the induced Euclidean metric is
$\ge d\rx^{2}+d\ry^{2}$ and $A_{E}=\nabla^{2}\phi_{j}/\sqrt{1+|\nabla\phi_{j}|^{2}}$ in the
coordinate frame, so $|A_{E}|\le|\nabla^{2}\phi_{j}|\le Cm^{4}\tau_{1}\le1$ by Lemma
\ref{lem:initialtor}. On the $i$-th catenoidal region take $F=\widehat\kappa_{i}$, the Euclidean
catenoid of waist radius $\tau_{i}$, whose principal curvatures are $\pm\tau_{i}^{-1}\sech^{2}t$, so
$|A_{E}|=\sqrt2\,\tau_{i}^{-1}\sech^{2}t\le\sqrt2\,\tau_{i}^{-1}\sech t$; since
$\tau_{i}^{-1}\sech t\ge\tau_{i}^{-1}\sech a_{i}=R^{-1}\ge1$ on $\bK_{a_{i}}$, the constant $1$ is
absorbed.
\end{proof}

\subsection{The minimal surfaces}\label{ss:minimal}
The main theorem of \cite{Wiygul} produces, for $m$ large, parameters $\zeta,\xi$ bounded by
$c[N,k,\ell]$ and a smooth $\cG$-invariant function $u$ on $\Sigma^{0}=\Sigma^{0}[N,k,\ell,m,\zeta,\xi]$
whose normal graph
\begin{equation}\label{eq:Psi}
\Psi(p):=\exp_{p}\bigl(u(p)\,\nu^{0}(p)\bigr),\qquad p\in\Sigma^{0},
\end{equation}
$\nu^{0}$ being the unit normal of $\Sigma^{0}$, is a closed embedded minimal surface invariant
under $\cG$ \cite[Theorem 1.1, Lemma 4.42, Theorem 6.50]{Wiygul}; the action of $\cG$ on normal perturbations carries no sign because every element of $\cG$ preserves each side of $\Sigma^{0}$ \cite[(4.4)]{Wiygul}. We now record the precise form of the estimate on $u$ and the three consequences of it which will be used.

\emph{Weighted norms.} For a function $v$ on $\Sigma^{0}$, a Riemannian metric $h$ on $\Sigma^{0}$
and a positive weight $f$, the weighted H\"older norm of \cite[(4.6)]{Wiygul} is
\begin{equation}\label{eq:wnorm}
\|v:C^{2,\alpha}(\Sigma^{0},h,f)\|:=\sup_{p\in\Sigma^{0}}f(p)^{-1}\,\|v:C^{2,\alpha}(B[p,1,h],h)\|,
\end{equation}
where $B[p,1,h]$ is the $h$-ball of radius $1$ about $p$ and the norm on the ball is the sum of the
$C^{2}$ norm and of the H\"older seminorm of the second derivatives, all measured with $h$. Let
$\rho$ be the $\cG$-invariant conformal factor of \cite[(4.11)--(4.14)]{Wiygul} and let
$\chi:=\rho^{2}g^{0}$ \cite[(4.9)]{Wiygul}. The \emph{decay norms} of \cite[\S4]{Wiygul}, paragraph
``Decay norms and a global estimate of the mean curvature'', are defined there, for
$\alpha\in(0,1)$ and $\gamma\in[0,\infty)$, by
\begin{equation}\label{eq:decaynorm}
\|v\|_{2,\alpha,\gamma}:=\bigl\|v:C^{2,\alpha}\bigl(\Sigma^{0},\chi,m^{\gamma}\rho^{-\gamma}\bigr)\bigr\| .
\end{equation}
Since the $C^{2,\alpha}$ norm on the unit $\chi$-ball about $p$ dominates $|v(p)|+|dv|_{\chi}(p)$,
\eqref{eq:wnorm} and \eqref{eq:decaynorm} give the pointwise bound
\begin{equation}\label{eq:normbound}
|v(p)|+|dv|_{\chi}(p)\ \le\ m^{\gamma}\rho(p)^{-\gamma}\,\|v\|_{2,\alpha,\gamma}\qquad(p\in\Sigma^{0}),
\end{equation}
and, since $\chi=\rho^{2}g^{0}$ and $\chi^{-1}=\rho^{-2}(g^{0})^{-1}$ on covectors,
\begin{equation}\label{eq:dug}
|dv|_{g^{0}}=\rho\,|dv|_{\chi}\qquad\text{on }\Sigma^{0}.
\end{equation}

\emph{The conformal factor.} The function $\rho$ is explicit. By \cite[(4.11)--(4.13)]{Wiygul} it
equals the constant $m$ at distance at least $2R$ from all removal points and, in polar coordinates
$r$ about a removal point, $m+(r^{-1}-m)\,\psi[2R,R](r)$ on the transition annulus $R\le r\le2R$,
while on the catenoidal regions \cite[(4.27)]{Wiygul} $\rho\circ\kappa_{i}=\tau_{i}^{-1}\sech t$.
Since $r^{-1}\ge(2R)^{-1}=5\ell m\ge m$ on the annulus,
\begin{equation}\label{eq:rho}
m\le\rho\le R^{-1}=10\ell m\ \text{ on }\cT^{0}[j],\qquad
\rho\circ\kappa_{i}=\tau_{i}^{-1}\sech t\ \text{ on }\bK_{a_{i}} .
\end{equation}
On $\bK_{a_{i}}$ one has $\tau_{i}^{-1}\sech t\ge\tau_{i}^{-1}\sech a_{i}=R^{-1}\ge m$ by
\eqref{eq:ai}, and $\tau_{i}^{-1}\sech t\le\tau_{i}^{-1}\le C\tau_{1}^{-1}$ by the remark after
\eqref{eq:heights}. Hence, for $m\ge m_{0}$,
\begin{equation}\label{eq:rhobounds}
m\ \le\ \rho\ \le\ C\,\tau_{1}^{-1}\qquad\text{on }\Sigma^{0},\qquad C=C[N,k,\ell],
\end{equation}
in agreement with \cite[(4.15)]{Wiygul}.

\emph{The final estimate.} The function $u$ of \cite[Theorem 6.50]{Wiygul} satisfies
\cite[(6.61)]{Wiygul}
\begin{equation}\label{eq:660}
\|u\|_{2,\alpha,\gamma}\ \le\ 2C_{1}\tau_{1},\qquad C_{1}=C_{1}[N,k,\ell],
\end{equation}
where $\alpha,\gamma\in(0,1)$ are the exponents fixed in the hypotheses of
\cite[Theorem 6.50]{Wiygul}; the linear theory of \cite[\S5]{Wiygul} is available for every such pair.
We stress that $u$ is a function on $\Sigma^{0}[N,k,\ell,m,\zeta,\xi]$ itself, and that the norm in
\eqref{eq:660} is the decay norm \eqref{eq:decaynorm} of that surface: in the proof of
\cite[Theorem 6.50]{Wiygul} one has $u=u_{1}+\mathcal P^{-1}v$, where $u_{1}$ is the first-order
solution of \cite[Corollary 5.66]{Wiygul} on $\Sigma^{0}[N,k,\ell,m,\zeta,\xi]$ and $v$ is a fixed point
living on the reference surface $\Sigma^{0}[N,k,\ell,m,0,0]$, transported by the explicit
$\cG$-equivariant diffeomorphism of \cite[(6.42)--(6.44)]{Wiygul}, whose norm is bounded in \cite[Lemma 6.45]{Wiygul}; the norm estimate is obtained in \cite[(6.57)]{Wiygul} and recorded for the fixed point in \cite[(6.61)]{Wiygul}. Of \eqref{eq:660} we shall use only the pointwise information \eqref{eq:normbound} at the level of first derivatives, and of the weight in \eqref{eq:decaynorm} only the positivity of $\gamma$; the specific values of $\alpha$ and $\gamma$
play no role (Lemma \ref{lem:pointwise}). We retain from all this the following property.

\begin{definition}\label{def:stacked}
Let $N\ge2$, $k,\ell\ge1$ and $m\ge m_{0}[N,k,\ell]$, and let
$\Sigma^{0}=\Sigma^{0}[N,k,\ell,m,\zeta,\xi]$ be an initial surface with $|\zeta|,|\xi|\le c[N,k,\ell]$.
A closed embedded minimal surface $\Sigma\subset\bS^{3}$ is a \emph{stacked Clifford torus of type
$(N,k,\ell,m)$} if it is $\cG[k,\ell,m]$-invariant and there is a $\cG$-equivariant diffeomorphism
$\Psi:\Sigma^{0}\to\Sigma$ such that the induced metrics $g^{0}$ of $\Sigma^{0}$ and $g$ of $\Sigma$
satisfy
\begin{equation*}
\bigl(1-m^{-2}\bigr)\,g^{0}\ \le\ \Psi^{*}g\ \le\ \bigl(1+m^{-2}\bigr)\,g^{0}
\qquad\text{as quadratic forms on }T\Sigma^{0}. \tag{W}\label{eq:W}
\end{equation*}
\end{definition}

The verification that the surfaces of \cite{Wiygul} have this property rests on two lemmas: a
formula for the metric of a normal graph, and a pointwise bound derived from \eqref{eq:660}.

\begin{lemma}[Metric of a normal graph]\label{lem:graphmetric}
Let $M\subset\bS^{3}$ be a closed embedded surface with induced metric $g^{0}$, unit normal
$\nu^{0}$, second fundamental form $A^{0}$ and shape operator $S$, $g^{0}(SX,Y)=A^{0}(X,Y)$, and let
$u\in C^{\infty}(M)$. The map $\Psi(p):=\exp_{p}(u(p)\nu^{0}(p))$ satisfies
\begin{equation}\label{eq:pullback}
\Psi^{*}g=\cos^{2}u\;g^{0}-\sin(2u)\,A^{0}+\sin^{2}u\;g^{0}(S\cdot,S\cdot)+du\otimes du ,
\end{equation}
$g$ being the round metric. Let $\|S\|$ denote the operator norm of $S$ with respect to $g^{0}$ and
put $s:=|u|+|u|\,\|S\|+|du|_{g^{0}}$. At every point where $s\le1$,
\[
(1-5s)\,g^{0}\ \le\ \Psi^{*}g\ \le\ (1+5s)\,g^{0}\qquad\text{as quadratic forms.}
\]
\end{lemma}

\begin{proof}
Regard $\bS^{3}\subset\bR^{4}$, write $p$ for the position vector and $\nu^{0}(p)\in T_{p}\bS^{3}$
for the unit normal. The geodesics of $\bS^{3}$ being $\tau\mapsto\cos\tau\,p+\sin\tau\,\nu$, we have
$\Psi(p)=\cos u(p)\,p+\sin u(p)\,\nu^{0}(p)$. For $X\in T_{p}M$, differentiating
$\langle\nu^{0},p\rangle=0$ and $\langle\nu^{0},\nu^{0}\rangle=1$ gives
$\langle D_{X}\nu^{0},p\rangle=-\langle\nu^{0},X\rangle=0$ and $\langle D_{X}\nu^{0},\nu^{0}\rangle=0$,
so $D_{X}\nu^{0}\in T_{p}M$, and $\langle D_{X}\nu^{0},Y\rangle=-\langle\nu^{0},D_{X}Y\rangle=-A^{0}(X,Y)$
for $Y\in T_{p}M$: thus $D_{X}\nu^{0}=-SX$. Hence
\[
d\Psi(X)=\cos u\,X-\sin u\,SX+du(X)\,e,\qquad e:=-\sin u\,p+\cos u\,\nu^{0},
\]
and $e$ is a unit vector orthogonal to $T_{p}M$. Expanding $|d\Psi(X)|^{2}$ gives
\eqref{eq:pullback}. Using $1-\cos^{2}u\le u^{2}$, $|\sin2u|\le2|u|$ and $\sin^{2}u\le u^{2}$,
\eqref{eq:pullback} gives, for every $X$,
\[
\bigl|\Psi^{*}g(X,X)-g^{0}(X,X)\bigr|\le\bigl(u^{2}+2|u|\,\|S\|+u^{2}\|S\|^{2}+|du|_{g^{0}}^{2}\bigr)\,g^{0}(X,X)
\le5s\,g^{0}(X,X)\qquad\text{if }s\le1 ,
\]
since then $u^{2}\le|u|\le s$, $2|u|\,\|S\|\le2s$, $u^{2}\|S\|^{2}\le s^{2}\le s$ and
$|du|_{g^{0}}^{2}\le s^{2}\le s$.
\end{proof}

\begin{lemma}[Pointwise smallness from the weighted estimate]\label{lem:pointwise}
Let $\Sigma^{0}=\Sigma^{0}[N,k,\ell,m,\zeta,\xi]$ with $|\zeta|,|\xi|\le c[N,k,\ell]$ and $m\ge m_{0}$,
and let $u$ be a smooth function on $\Sigma^{0}$ satisfying
\[
\bigl\|u:C^{2,\alpha}\bigl(\Sigma^{0},\chi,m^{\gamma}\rho^{-\gamma}\bigr)\bigr\|\ \le\ K\,\tau_{1}
\]
for some $\alpha\in(0,1)$, $\gamma>0$ and $K>0$. Then, with $s$ as in Lemma \ref{lem:graphmetric}
for $M=\Sigma^{0}$,
\[
\sup_{\Sigma^{0}}s\ \le\ C\,K\,m\,\tau_{1}^{\min\{\gamma,1\}},\qquad C=C[N,k,\ell].
\]
\end{lemma}

\begin{proof}
Let $p\in\Sigma^{0}$ and $f:=m^{\gamma}\rho^{-\gamma}$. By \eqref{eq:normbound} and
\eqref{eq:dug}, $|u(p)|\le f(p)K\tau_{1}$ and $|du|_{g^{0}}(p)\le\rho(p)f(p)K\tau_{1}$. By Lemma
\ref{lem:initialA} and \eqref{eq:rho}, $\|S\|\le|A^{0}|\le C_{A}\,\rho$ on all of $\Sigma^{0}$ with
$C_{A}=C_{A}[N,k,\ell]$: on the toral regions $|A^{0}|\le C\le C\rho$ because $\rho\ge m\ge1$, and
on the catenoidal regions $|A^{0}|\le C\tau_{i}^{-1}\sech t=C\rho$. Hence, using $\rho\ge1$,
\[
s(p)\le\bigl(1+C_{A}\rho(p)+\rho(p)\bigr)f(p)K\tau_{1}\le(2+C_{A})\,\rho(p)f(p)\,K\tau_{1},
\qquad\rho f=m^{\gamma}\rho^{1-\gamma}.
\]
If $\gamma\ge1$, then $\rho^{1-\gamma}\le m^{1-\gamma}$ by the lower bound in \eqref{eq:rhobounds},
so $\rho f\le m$ and $s(p)\le(2+C_{A})Km\tau_{1}$. If $0<\gamma<1$, then
$\rho^{1-\gamma}\le C^{1-\gamma}\tau_{1}^{\gamma-1}$ by the upper bound in \eqref{eq:rhobounds}, so
$\rho f\,\tau_{1}\le C\,m^{\gamma}\tau_{1}^{\gamma}\le C\,m\,\tau_{1}^{\gamma}$ and
$s(p)\le(2+C_{A})CKm\tau_{1}^{\gamma}$. In both cases $s(p)\le CKm\tau_{1}^{\min\{\gamma,1\}}$.
\end{proof}

\begin{corollary}[The surfaces of \cite{Wiygul} are stacked Clifford tori]\label{cor:wiygulW}
Let $\Sigma$ be the normal graph \eqref{eq:Psi} over $\Sigma^{0}=\Sigma^{0}[N,k,\ell,m,\zeta,\xi]$,
$|\zeta|,|\xi|\le c[N,k,\ell]$, of the smooth $\cG$-invariant function $u$ produced by
\cite[Theorem 6.50]{Wiygul}, which satisfies \eqref{eq:660}. Then, for $m\ge m_{0}[N,k,\ell]$,
$\Sigma$ is a stacked Clifford torus of type $(N,k,\ell,m)$, with $\Psi$ as in \eqref{eq:Psi}.
\end{corollary}

\begin{proof}
\emph{The estimate.} By Lemma \ref{lem:pointwise} with $K=2C_{1}$,
$\sup_{\Sigma^{0}}s\le Cm\tau_{1}^{\min\{\gamma,1\}}$. By \eqref{eq:tau} and the remark after
\eqref{eq:heights}, $\tau_{1}\le\exp(-cm^{2})$ with $c=c[N,k,\ell]>0$ for $m\ge m_{0}$, so
$5\sup s\le5Cm\exp(-c\min\{\gamma,1\}m^{2})\le m^{-2}$ after enlarging $m_{0}$. In particular $s\le1$
everywhere, and Lemma \ref{lem:graphmetric} gives \eqref{eq:W}.

\emph{Equivariance and bijectivity.} For $\mathfrak g\in\cG$ the function $u$ is invariant and, by
\cite[(4.4)]{Wiygul}, $\mathfrak g$ preserves each side of $\Sigma^{0}$, that is
$\mathfrak g_{*}\nu^{0}(p)=\nu^{0}(\mathfrak gp)$; as $\mathfrak g$ is linear on $\bR^{4}$,
$\Psi(\mathfrak gp)=\cos u(p)\,\mathfrak gp+\sin u(p)\,\mathfrak g\nu^{0}(p)=\mathfrak g\Psi(p)$.
By \eqref{eq:W}, $\Psi$ is an immersion of the closed surface $\Sigma^{0}$. The map $\Psi$ is the
normal deformation $\iota[u]$ of \cite[(5.1)]{Wiygul}, and the last paragraph of the proof of
\cite[Theorem 6.50]{Wiygul} shows that $\iota[u]$ is an embedding for $m\ge m_{0}$: the decay of $u$
toward the waists built into \eqref{eq:decaynorm} keeps $|u|$ below a small multiple of $\tau_{1}$ on
the catenoidal cores, hence below the normal injectivity radius of $\Sigma^{0}$ there. Thus $\Psi$ is
injective, and it is a diffeomorphism onto its image, the embedded surface $\Sigma$. Finally $\Sigma$
is closed, embedded, minimal and $\cG$-invariant by the same theorem.
\end{proof}

\begin{remark}\label{rem:W}
Whether the fixed point of \cite[Theorem 6.50]{Wiygul} is unique plays no role: Theorem A applies
to every surface with the property of Definition \ref{def:stacked}, and Corollary \ref{cor:wiygulW}
applies to every fixed point. The bound \eqref{eq:660} is far stronger than needed: the metric
distortion it produces is of order $m\tau_{1}^{\min\{\gamma,1\}}$, whereas \eqref{eq:W} only asks for
$m^{-2}$. Moreover Lemma \ref{lem:pointwise} uses of the weight in \eqref{eq:decaynorm} only the
positivity of the exponent $\gamma$, through the two-sided bound \eqref{eq:rhobounds}: for
$\gamma=0$ the argument would give no more than $s\le C\rho\tau_{1}$, which is of order one on the
part $|t|\lesssim\ln m$ of each tunnel. The decay of $u$ toward the waists built into
\eqref{eq:decaynorm} is what makes the sharp comparison \eqref{eq:W} available on all of $\Sigma$, and
we use it in that form throughout. (We expect the leading value $4$ of Theorem B to survive a bounded
distortion of the metric on a portion of each tunnel of conformal length $O(\ln m)$, since such a
portion would contribute only $O(m^{-2}\ln m)$ to the reciprocal of the conductance of a tunnel of
conformal length $2a_{i}\sim m^{2}$; we do not pursue this, as \eqref{eq:W} is available.)
\end{remark}

\subsection{Decomposition and comparison}
Let $\Sigma$ be a stacked Clifford torus of type $(N,k,\ell,m)$ with $\Psi$ as in
\eqref{eq:W}. We transport the decomposition of $\Sigma^{0}$ to $\Sigma$:
\begin{equation}\label{eq:decomp}
\Sigma=\bigcup_{j=1}^{N}\cT[j]\ \cup\ \bigcup_{i=1}^{N-1}\bigcup_{s=1}^{M}\cK_{i}^{(s)},
\qquad
\cT[j]:=\Psi(\cT^{0}[j]),\quad \cK_{i}^{(s)}:=\Psi\bigl(\mathfrak g_{s}\kappa_{i}(\bK_{a_{i}})\bigr),
\end{equation}
where $\mathfrak g_{1}=\mathrm{id},\mathfrak g_{2},\dots,\mathfrak g_{M}\in\cG$ carry the removal point
$(\rx_{i},\ry_{i})$ onto the $M$ points of $\mathsf L_{(i+1)\bmod2}$. The pieces have pairwise disjoint
interiors and are glued along the end circles $\Gamma_{i,\pm}^{(s)}:=\Psi(\mathfrak g_{s}\kappa_{i}(\{t=\pm a_{i}\}))$.
Each tunnel $\cK_{i}^{(s)}$ carries the coordinates $(t,\theta)\in\bK_{a_{i}}$ through
$\Psi\circ\mathfrak g_{s}\circ\kappa_{i}$, and we write $\cK_{i}^{(s),\pm}$ for its halves
$\{\pm t\ge0\}$. Each toral region $\cT[j]$ carries the chart $\pi_{j}\circ\Psi^{-1}:\cT[j]\to P_{j}$.
By Lemmas \ref{lem:initialcat}, \ref{lem:initialtor} and \eqref{eq:W}, for $m\ge m_{0}$ (so that
$Cm^{2}\tau_{1}\le m^{-2}$),
\begin{equation}\label{eq:comp}
\begin{aligned}
(1-\varepsilon_{m})\,\tau_{i}^{2}\cosh^{2}t\,\widehat\chi&\le g\le(1+\varepsilon_{m})\,\tau_{i}^{2}\cosh^{2}t\,\widehat\chi
&&\text{on }\cK_{i}^{(s)},\\
(1-\varepsilon_{m})\,(d\rx^{2}+d\ry^{2})&\le g\le(1+\varepsilon_{m})\,(d\rx^{2}+d\ry^{2})
&&\text{on }\cT[j],
\end{aligned}
\qquad \varepsilon_{m}:=3m^{-2}.
\end{equation}
All these identifications are $\cG$-equivariant, $\cG$ preserving each $\cT[j]$ and permuting the
tunnels of each layer among themselves.

\begin{convention}\label{conv:avg}
For an end circle $\Gamma=\Gamma_{i,\pm}^{(s)}$ and a function $u$ on $\Sigma$ we write
\[
\avg_{\Gamma}u:=\frac{1}{2\pi}\int_{0}^{2\pi}u(\pm a_{i},\theta)\dd\theta
\]
for the mean with respect to the coordinate $\theta$; by Lemma \ref{lem:initialtor} this is also the
mean over the circle $\partial B(p,R)$ of $P_{j}$ with respect to the polar angle. For a region
$\Omega\subset\Sigma$ we write $\avg_{\Omega}u:=|\Omega|^{-1}\int_{\Omega}u\dd A$ with respect to the
area of $\Sigma$.
\end{convention}

\begin{lemma}[Transfer]\label{lem:transfer}
Let $h$ and $h'$ be Riemannian metrics on a surface $S$ with $(1-\varepsilon)h\le h'\le(1+\varepsilon)h$,
$0<\varepsilon\le\tfrac12$. Then for every $f\in H^{1}(S)$
\begin{gather*}
\frac{1-\varepsilon}{1+\varepsilon}\int_{S}|\nabla f|_{h}^{2}\dd A_{h}\le\int_{S}|\nabla f|_{h'}^{2}\dd A_{h'}
\le\frac{1+\varepsilon}{1-\varepsilon}\int_{S}|\nabla f|_{h}^{2}\dd A_{h},\\
(1-\varepsilon)\dd A_{h}\le\dd A_{h'}\le(1+\varepsilon)\dd A_{h},
\end{gather*}
and in dimension two the Dirichlet integral $\int|\nabla f|_{h}^{2}dA_{h}$ is invariant under
conformal changes of $h$. In particular all comparisons in \eqref{eq:comp} affect Dirichlet
integrals, $L^{2}$ norms and areas by factors $1+O(m^{-2})$.
\end{lemma}

\begin{proof}
If $h'\le(1+\varepsilon)h$ then $h'^{-1}\ge(1+\varepsilon)^{-1}h^{-1}$ on covectors, so
$|\nabla f|^{2}_{h'}\ge(1+\varepsilon)^{-1}|\nabla f|^{2}_{h}$, while $dA_{h'}\ge(1-\varepsilon)dA_{h}$
from $h'\ge(1-\varepsilon)h$; the other inequalities are symmetric. Conformal invariance in
dimension two is the identity $|\nabla f|^{2}_{\rho^{2}h}dA_{\rho^{2}h}=|\nabla f|^{2}_{h}dA_{h}$.
\end{proof}

\section{Balancing, the line graph and the spectral gap}\label{sec:balancing}

Let $P_{N}$ be the path graph with vertices $1,\dots,N$ and edges $e_{i}=\{i,i+1\}$, $1\le i\le N-1$;
let $A(P_{N})$ be its adjacency matrix and $L(P_{N})=\operatorname{diag}(\deg)-A(P_{N})$ its
combinatorial Laplacian, whose quadratic form is $c\mapsto\sum_{i=1}^{N-1}(c_{i}-c_{i+1})^{2}$. The
line graph of $P_{N}$, whose vertices are the edges $e_{1},\dots,e_{N-1}$ with $e_{i}$ adjacent to
$e_{i+1}$, is $P_{N-1}$. Let $D$ be the $N\times(N-1)$ incidence matrix with the consistent
orientation, $D_{v,e_{i}}=-1$ if $v=i$, $D_{v,e_{i}}=+1$ if $v=i+1$, and $D_{v,e_{i}}=0$ otherwise.

\begin{lemma}[Incidence identity]\label{lem:incidence}
$L(P_{N})=DD^{\mathsf T}$ and $D^{\mathsf T}D=2I_{N-1}-A(P_{N-1})$. The matrix $D$ has rank $N-1$,
the matrix $D^{\mathsf T}D$ is positive definite, and its spectrum, with multiplicities, is the
spectrum of $L(P_{N})$ with the eigenvalue $0$ removed. Consequently
\[
\lambda_{1}\bigl(L(P_{N})\bigr)=2-\lambda_{\max}\bigl(A(P_{N-1})\bigr),\qquad
\lambda_{\max}\bigl(L(P_{N})\bigr)=2-\lambda_{\min}\bigl(A(P_{N-1})\bigr).
\]
\end{lemma}

\begin{proof}
$(DD^{\mathsf T})_{vw}=\sum_{e}D_{ve}D_{we}$ equals the number of edges at $v$ if $v=w$, equals
$(-1)(+1)=-1$ if $\{v,w\}$ is an edge, and vanishes otherwise; this is $L(P_{N})$.
$(D^{\mathsf T}D)_{ef}=\sum_{v}D_{ve}D_{vf}$ equals $2$ if $e=f$; if $e=e_{i}$ and $f=e_{i+1}$ the only
common vertex is $i+1$, which is the head of $e_{i}$ and the tail of $e_{i+1}$, so the entry is
$(+1)(-1)=-1$; and it vanishes if $e$ and $f$ are not adjacent. Thus $D^{\mathsf T}D=2I-A(P_{N-1})$.
Next, $D^{\mathsf T}c=0$ means $c_{i+1}=c_{i}$ for all $i$, so $\ker D^{\mathsf T}$ consists of the
constant vectors and $\operatorname{rank}D=N-1$; hence $\ker(DD^{\mathsf T})=\ker D^{\mathsf T}$ is
one-dimensional and $D^{\mathsf T}D$ is positive definite. For any real matrix the nonzero eigenvalues
of $DD^{\mathsf T}$ and $D^{\mathsf T}D$ coincide with multiplicities (they are the squares of the
nonzero singular values of $D$); as $DD^{\mathsf T}$ has exactly $N-1$ nonzero eigenvalues counted
with multiplicity, these are all the eigenvalues of the $(N-1)\times(N-1)$ matrix $D^{\mathsf T}D$.
The last two identities follow.
\end{proof}

\begin{lemma}[Spectra of paths]\label{lem:pathspec}
\begin{enumerate}
\item[(i)] For $M\ge1$ the eigenvalues of $A(P_{M})$ are $2\cos\bigl(j\pi/(M+1)\bigr)$,
$j=1,\dots,M$, with eigenvectors $v^{(j)}=\bigl(\sin(ij\pi/(M+1))\bigr)_{i=1}^{M}$; for $j=1$ the
eigenvector is positive.
\item[(ii)] The eigenvalues of $L(P_{N})$ are $2\bigl(1-\cos(j\pi/N)\bigr)$, $j=0,\dots,N-1$, with
eigenvectors $c^{(j)}=\bigl(\cos\tfrac{(2r-1)j\pi}{2N}\bigr)_{r=1}^{N}$. In particular
$\lambda_{1}(L(P_{N}))=2(1-\cos(\pi/N))$ is attained at $c^{(1)}$, which satisfies
$c^{(1)}_{N+1-r}=-c^{(1)}_{r}$ and $\sum_{r}c^{(1)}_{r}=0$.
\end{enumerate}
\end{lemma}

\begin{proof}
(i) With $x_{i}:=ij\pi/(M+1)$ and $y:=j\pi/(M+1)$, the identity $\sin(x_{i}-y)+\sin(x_{i}+y)=2\sin x_{i}\cos y$
gives $(Av^{(j)})_{i}=v^{(j)}_{i-1}+v^{(j)}_{i+1}=2\cos y\,v^{(j)}_{i}$ for $1<i<M$, and also for
$i=1$ and $i=M$ because $\sin(0)=0=\sin(j\pi)$. The $M$ numbers $2\cos(j\pi/(M+1))$ are distinct, so
these are all the eigenvalues, and $v^{(1)}_{i}=\sin(i\pi/(M+1))>0$.

(ii) With $x_{r}:=(2r-1)j\pi/(2N)$ and $y:=j\pi/N$, the identity $\cos(x_{r}-y)+\cos(x_{r}+y)=2\cos x_{r}\cos y$
gives $-c^{(j)}_{r-1}+2c^{(j)}_{r}-c^{(j)}_{r+1}=2(1-\cos y)c^{(j)}_{r}$ for $1<r<N$. At $r=1$ the
row of $L(P_{N})$ is $(1,-1,0,\dots)$ and the same identity applies because
$c^{(j)}_{0}=\cos(-j\pi/2N)=c^{(j)}_{1}$; at $r=N$ it applies because
$c^{(j)}_{N+1}=\cos(j\pi+j\pi/2N)=(-1)^{j}\cos(j\pi/2N)=c^{(j)}_{N}$. The $N$ eigenvalues
$2(1-\cos(j\pi/N))$, $0\le j\le N-1$, are distinct, so the $c^{(j)}$ form an eigenbasis. Finally
$c^{(1)}_{N+1-r}=\cos\bigl(\pi-\tfrac{(2r-1)\pi}{2N}\bigr)=-c^{(1)}_{r}$, whence $\sum_{r}c^{(1)}_{r}=0$.
\end{proof}

\begin{lemma}[Perron structure of the limiting system]\label{lem:perron}
Let $N\ge2$ and let $b=(b_{1},\dots,b_{N-1})$ be given by $b_{1}=1$, by a solution
$(b_{2},\dots,b_{n})$ of the limiting system with all $b_{i}>0$ when $N\ge4$, by the conventions
$b_{2}=0$ $(N=2)$, $b_{2}=1$ $(N=3)$, and by $b_{i}=b_{N-i}$ for $n<i\le N-1$. Then $b$ is a positive
vector and
\[
A(P_{N-1})\,b=b_{2}\,b .
\]
\end{lemma}

\begin{proof}
Set $b_{0}:=b_{N}:=0$, so that $(A(P_{N-1})b)_{i}=b_{i-1}+b_{i+1}$ for $1\le i\le N-1$. For $N=2$ the
graph $P_{1}$ has a single vertex, $A(P_{1})=0$, $b=(1)$ and $b_{2}=0$. For $N=3$, $b=(1,1)$ and
$A(P_{2})b=(1,1)=b_{2}b$. Let $N\ge4$, so $n\ge2$. For $i=1$: $b_{0}+b_{2}=b_{2}=b_{2}b_{1}$. For
$2\le i\le n-1$: $b_{i-1}+b_{i+1}=b_{2}b_{i}$ is the first line of the limiting system. For $i=n$:
if $N=2n$ then $b_{n+1}=b_{N-n-1}=b_{n-1}$ and the second line reads $2b_{n-1}=b_{2}b_{n}$; if
$N=2n+1$ then $b_{n+1}=b_{N-n-1}=b_{n}$ and the second line reads $b_{n-1}+b_{n}=b_{2}b_{n}$. For
$n<i\le N-1$ we have $1\le N-i\le n$ and, by the symmetry $b_{i}=b_{N-i}$,
$b_{i-1}+b_{i+1}=b_{N-i+1}+b_{N-i-1}=b_{2}b_{N-i}=b_{2}b_{i}$ by the cases already treated.
Positivity holds because $b_{1}=1$ and $b_{i}>0$ for $2\le i\le n$.
\end{proof}

\begin{proposition}[Closed form of the balancing constants]\label{prop:b2}
Let $N\ge2$ and let $d_{i}[N]$ be the limiting waist ratios of {\rm(B1)}, so that $b_{i}[N]=d_{i}[N]$ in the notation of Theorem B, with $d_{2}[2]=0$ and
$d_{2}[3]=1$, extended by $d_{1}=1$ and $d_{i}=d_{N-i}$. Then
\begin{gather*}
d_{i}[N]=\frac{\sin(i\pi/N)}{\sin(\pi/N)}\quad(1\le i\le N-1),\\
b_{2}[N]:=d_{2}[N]=2\cos\frac{\pi}{N}=\lambda_{\max}\bigl(A(P_{N-1})\bigr),\qquad
\lambda_{1}\bigl(L(P_{N})\bigr)=2-b_{2}[N].
\end{gather*}
In particular $b_{2}[4]=\sqrt2$ and $b_{2}[5]=\tfrac{1+\sqrt5}{2}$.
\end{proposition}

\begin{proof}
For $N=2$ and $N=3$ the statements are immediate from $A(P_{1})=0$ and $\operatorname{spec}A(P_{2})=\{\pm1\}$.
Let $N\ge4$. By Lemma \ref{lem:perron} the vector $d=(d_{1},\dots,d_{N-1})$ is a positive eigenvector
of $A(P_{N-1})$ with eigenvalue $d_{2}$. The matrix $A(P_{N-1})$ is nonnegative and irreducible,
$P_{N-1}$ being connected, so by the Perron--Frobenius theorem it has, up to positive multiples, a
unique positive eigenvector, whose eigenvalue is the spectral radius $\lambda_{\max}(A(P_{N-1}))$.
Hence $d_{2}=\lambda_{\max}(A(P_{N-1}))=2\cos(\pi/N)$ by Lemma \ref{lem:pathspec}(i) with $M=N-1$, and
$d$ is the positive multiple of $\bigl(\sin(i\pi/N)\bigr)_{i}$ with first entry $1$. The identity for
$\lambda_{1}(L(P_{N}))$ is Lemma \ref{lem:incidence}. The special values are
$2\cos(\pi/4)=\sqrt2$ and $2\cos(\pi/5)=(1+\sqrt5)/2$.
\end{proof}

\begin{corollary}\label{cor:rate}
For $m\ge m_{0}[N]$ the constants of \eqref{eq:system} satisfy
\[
|b_{i}[N,k,\ell,m]-d_{i}[N]|\le\frac{C}{m^{2}}\quad(1\le i\le N-1),\qquad
2-b_{2}[N,k,\ell,m]\ \ge\ 1-\cos\frac{\pi}{N}\ >0,
\]
and all $b_{i}$ lie in a compact subset of $(0,\infty)$ depending only on $N$.
\end{corollary}

\begin{proof}
For $N\in\{2,3\}$ the constants are exact. For $N=4$ and $N=5$ we have $n=2$ and \eqref{eq:system} reduces to the single scalar equation $b_{2}^{2}-2=-\varepsilon b_{2}\ln b_{2}$,
respectively $b_{2}^{2}-b_{2}-1=-\varepsilon b_{2}\ln b_{2}$, with $\varepsilon:=8\pi/(k\ell m^{2})$; at $\varepsilon=0$ the roots are $d_{2}[4]=\sqrt2$ and $d_{2}[5]=\frac{1+\sqrt5}{2}$, and the derivatives $2\sqrt2$ and $\sqrt5$ are nonzero, so the implicit function theorem gives $|b_{2}-d_{2}[N]|=O(\varepsilon)=O(m^{-2})$ in both cases. Let now $N\ge6$, let $A_{N}(\beta)$ be the $(n-1)\times(n-1)$ matrix of \cite[(3.19)--(3.20)]{Wiygul}, and put $F(x):=A_{N}(x_{2})x-e_{1}$ and $G(x):=(x_{i}\ln x_{i})_{i=2}^{n}$ for $x=(x_{2},\dots,x_{n})$ with positive entries. The limiting system is $F(x)=0$, and \eqref{eq:system} is $F(x)+\varepsilon G(x)=0$ with $\varepsilon:=8\pi/(k\ell m^{2})$. By (B1), $F(d)=0$ and $dF(d)$ is invertible; the map $H(x,\varepsilon):=F(x)+\varepsilon G(x)$ is smooth near $(d,0)$ with $\partial_{x}H(d,0)=dF(d)$, so the implicit function theorem gives a smooth branch $x(\varepsilon)$
of solutions with $x(0)=d$, unique in a neighbourhood of $d$, and $|x(\varepsilon)-d|\le C[N]\varepsilon$ for small $\varepsilon$. By (B2), $b[N,k,\ell,m]=x(\varepsilon)$ for $m\ge m_{0}[N]$, which gives the first estimate for $2\le i\le n$, hence for all $i$ by the symmetry. Since $2-d_{2}[N]=2(1-\cos(\pi/N))$ by Proposition \ref{prop:b2}, the second estimate holds for $m$ large, and the last statement follows from $d_{i}[N]>0$.
\end{proof}

\begin{remark}[The heights form the Fiedler vector]\label{rem:fiedler}
With the consistent orientation, $(Db)_{v}=b_{v-1}-b_{v}$, and for the limiting ratios
$b_{i}=\sin(i\pi/N)/\sin(\pi/N)$ one computes
\[
(Db)_{v}=\frac{\sin\frac{(v-1)\pi}{N}-\sin\frac{v\pi}{N}}{\sin\frac\pi N}
=-\frac{\cos\frac{(2v-1)\pi}{2N}}{\cos\frac{\pi}{2N}}=-\frac{c^{(1)}_{v}}{\cos\frac{\pi}{2N}} .
\]
Thus $Db$ is the eigenvector of $L(P_{N})=DD^{\mathsf T}$ for its smallest nonzero eigenvalue, as it
must be: $DD^{\mathsf T}(Db)=D(D^{\mathsf T}Db)=(2-b_{2})Db$. In the linear model of the construction
the vector of heights of the tori is proportional to $Db$ \cite[(2.15)--(2.17)]{Wiygul}, so the
stacking configuration is, to leading order, the Fiedler vector of the chain. This is the geometric
reason why the mode which is constant on each torus is governed by the spectral gap and not by
another eigenvalue of $L(P_{N})$.
\end{remark}

\begin{remark}[Chebyshev polynomials]\label{rem:cheb}
In \cite[Lemma 3.18]{Wiygul} the number $d_{2}[N]$ is characterised as the largest root of a
polynomial $P_{n+1}$ defined by the recursion $P_{i+1}(\lambda)=\lambda P_{i}(\lambda)-P_{i-1}(\lambda)$
with initial data $P_{1}=2$, $P_{2}=\lambda$ for even $N$ and $P_{1}=1$, $P_{2}=\lambda-1$ for odd $N$
\cite[(3.22)--(3.25)]{Wiygul}. This is the Chebyshev recursion, and the substitution
$\lambda=2\cos\theta$ gives $P_{i}(2\cos\theta)=2\cos((i-1)\theta)$ in the even case and
$P_{i}(2\cos\theta)=\cos\bigl(\tfrac{(2i-1)\theta}{2}\bigr)/\cos\tfrac\theta2$ in the odd case, as one
checks from $2\cos\theta\cos\alpha=\cos(\alpha+\theta)+\cos(\alpha-\theta)$ and the initial data.
The largest root of $P_{n+1}$ is therefore $2\cos\theta_{0}$ with $\theta_{0}=\pi/(2n)=\pi/N$ for
$N=2n$ and $\theta_{0}=\pi/(2n+1)=\pi/N$ for $N=2n+1$, in agreement with Proposition \ref{prop:b2},
which also explains the monotonicity of $d_{2}[N]$ in $N$ observed in \cite{Wiygul}. The closed
form of all the ratios $d_{i}[N]$, and their interpretation as the Perron vector of the line graph,
do not seem to have been recorded before.
\end{remark}

\section{The channel inequality and the conductances}\label{sec:channel}

\begin{lemma}[Channel inequality]\label{lem:channel}
Let $a>0$, let $\bK_{a}=[-a,a]\times\bS^{1}$ with the flat metric $\widehat\chi=dt^{2}+d\theta^{2}$,
and let $\Gamma_{\pm}=\{t=\pm a\}$. For every $f\in H^{1}(\bK_{a})$,
\[
\int_{\bK_{a}}|\nabla f|^{2}\dd t\dd\theta\ \ge\ \frac{\pi}{a}\Bigl(\avg_{\Gamma_{+}}f-\avg_{\Gamma_{-}}f\Bigr)^{2},
\qquad \avg_{\Gamma_{\pm}}f:=\frac1{2\pi}\int_{0}^{2\pi}f(\pm a,\theta)\dd\theta,
\]
with equality if and only if $f(t,\theta)=\alpha+\beta t$ for constants $\alpha,\beta$.
\end{lemma}

\begin{proof}
Both sides are continuous on $H^{1}(\bK_{a})$, the traces on $\Gamma_{\pm}$ being continuous, so it
suffices to consider smooth $f$. For each $\theta$, $f(a,\theta)-f(-a,\theta)=\int_{-a}^{a}f_{t}\dd t$,
so by the Cauchy--Schwarz inequality
\begin{equation}\label{eq:CS1}
\bigl(f(a,\theta)-f(-a,\theta)\bigr)^{2}\le2a\int_{-a}^{a}f_{t}(t,\theta)^{2}\dd t .
\end{equation}
With $\delta:=\avg_{\Gamma_{+}}f-\avg_{\Gamma_{-}}f=\frac1{2\pi}\int_{0}^{2\pi}(f(a,\theta)-f(-a,\theta))\dd\theta$,
a second application gives
\begin{equation}\label{eq:CS2}
\delta^{2}\le\frac1{2\pi}\int_{0}^{2\pi}\bigl(f(a,\theta)-f(-a,\theta)\bigr)^{2}\dd\theta .
\end{equation}
Combining, $\delta^{2}\le\frac{a}{\pi}\int_{\bK_{a}}f_{t}^{2}\le\frac a\pi\int_{\bK_{a}}|\nabla f|^{2}$.
Equality forces $f_{t}$ to be constant in $t$ for almost every $\theta$ (equality in
\eqref{eq:CS1}), $f(a,\theta)-f(-a,\theta)$ to be constant in $\theta$ (equality in \eqref{eq:CS2}),
and $f_{\theta}\equiv0$ (equality in the last step); together these give $f=\alpha+\beta t$.
Conversely, for $f=t/a$ one has $\delta=2$ and $\int|\nabla f|^{2}=a^{-2}\cdot2a\cdot2\pi=(\pi/a)\delta^{2}$.
\end{proof}

\begin{corollary}\label{cor:channel}
Let $g$ be a metric on $\bK_{a}$ with $(1-\varepsilon)\rho^{2}\widehat\chi\le g\le(1+\varepsilon)\rho^{2}\widehat\chi$
for a positive function $\rho$ and some $\varepsilon\in(0,\tfrac12]$. Then for every $f\in H^{1}(\bK_{a})$
\[
\int_{\bK_{a}}|\nabla f|_{g}^{2}\dd A_{g}\ \ge\ \frac{1-\varepsilon}{1+\varepsilon}\cdot\frac{\pi}{a}
\Bigl(\avg_{\Gamma_{+}}f-\avg_{\Gamma_{-}}f\Bigr)^{2}.
\]
\end{corollary}

\begin{proof}
By Lemma \ref{lem:transfer} the Dirichlet integral of $g$ is at least $\frac{1-\varepsilon}{1+\varepsilon}$
times that of $\rho^{2}\widehat\chi$, which equals that of $\widehat\chi$ by conformal invariance.
\end{proof}

The channel inequality plays the role that the capacity of a tunnel would play, but it is an
inequality for arbitrary $H^{1}$ functions with no condition on the boundary data, and it is
sharp; its conformal invariance is what allows the geometry of a tunnel to enter only through the
single number $a$.

\begin{definition}\label{def:conductance}
We call $\pi/a$ the \emph{conductance} of a channel of conformal length $2a$, and for a stacked Clifford torus of type $(N,k,\ell,m)$, with $\Sigma^{0}$ and $\Psi$ as in Definition \ref{def:stacked}, we set
\[
C_{i}:=M\,\frac{\pi}{a_{i}}=\frac{k\ell m^{2}\pi}{a_{i}},\qquad 1\le i\le N-1,
\]
the total conductance of the $M$ tunnels joining the $i$-th and the $(i+1)$-st tori.
\end{definition}
The numbers $a_{i}$, and hence the $C_{i}$, depend on the parameters $\zeta$ of the initial surface $\Sigma^{0}$ witnessing Definition \ref{def:stacked}. Proposition \ref{prop:conductance} shows that this dependence affects $C_{i}$ only at relative order $O(m^{-2})$, uniformly for
$|\zeta|,|\xi|\le c[N,k,\ell]$, so that all statements below are independent of the choice.

\begin{proposition}\label{prop:conductance}
For $m\ge m_{0}$ and every $1\le i\le N-1$,
\[
C_{i}=\frac{8\pi^{2}}{2-b_{2}[N,k,\ell,m]}\bigl(1+O(m^{-2})\bigr)
=\frac{8\pi^{2}}{2-b_{2}[N]}\bigl(1+O(m^{-2})\bigr)
=\frac{4\pi^{2}}{1-\cos(\pi/N)}\bigl(1+O(m^{-2})\bigr).
\]
In particular $C_{i}$ is bounded above and below by positive constants depending only on $N$, the
total conductance is to leading order the same for all layers and independent of $k$, $\ell$ and
$m$, and $C_{N}:=8\pi^{2}/(2-b_{2}[N])$ satisfies $C_{N}=8N^{2}+O(1)$ as $N\to\infty$.
\end{proposition}

\begin{proof}
In \eqref{eq:aiexp} the terms $\zeta_{1}$, $\ln b_{i}$, $\zeta_{i}/(k\ell m^{2})$ and
$\ln(1+\sqrt{1-100\ell^{2}m^{2}\tau_{i}^{2}})\in[0,\ln2]$ are bounded by constants depending only on
$N,k,\ell$, so $a_{i}=\frac{k\ell m^{2}}{8\pi}(2-b_{2})+O(1)$, and since $2-b_{2}\ge1-\cos(\pi/N)$ by
Corollary \ref{cor:rate}, $a_{i}=\frac{k\ell m^{2}}{8\pi}(2-b_{2})(1+O(m^{-2}))$. Hence
$C_{i}=k\ell m^{2}\pi/a_{i}=8\pi^{2}(2-b_{2})^{-1}(1+O(m^{-2}))$. Corollary \ref{cor:rate} gives
$2-b_{2}=(2-b_{2}[N])(1+O(m^{-2}))$, and $2-b_{2}[N]=2(1-\cos(\pi/N))$ by Proposition \ref{prop:b2}.
Finally $1-\cos(\pi/N)=\frac{\pi^{2}}{2N^{2}}-\frac{\pi^{4}}{24N^{4}}+O(N^{-6})$ gives
$C_{N}=8N^{2}\bigl(1-\frac{\pi^{2}}{12N^{2}}+O(N^{-4})\bigr)^{-1}=8N^{2}+O(1)$.
\end{proof}

\begin{remark}[The regime $m\gg N$]\label{rem:regime}
Our theorem is a fixed-$N$ result and no uniformity in $N$ is claimed. For each fixed $N$, Proposition \ref{prop:b2}, Corollary \ref{cor:rate} and \eqref{eq:aiexp} give
\[
a_{i}
=
\frac{k\ell}{4\pi}\left(1-\cos\frac{\pi}{N}\right)m^{2}
+O_{N,k,\ell}(1)
\qquad (m\to\infty).
\]
Moreover,
\[
\frac{k\ell}{4\pi}\left(1-\cos\frac{\pi}{N}\right)
=
\frac{k\ell\pi}{8N^{2}}\bigl(1+O(N^{-2})\bigr)
\qquad (N\to\infty).
\]
Thus, at the level of the leading term, formally allowing $N$ to vary suggests the joint long-neck scale $\sqrt{k\ell}\,m/N\to\infty$. No estimate proved here is uniform in $N$, however, and in particular no two-parameter error estimate in $N$ and $m$ is asserted. The limiting graph value $4$ of Theorem B is independent of $N$.
\end{remark}

\section{The invariant sector}\label{sec:sector}

Throughout this section $\Sigma$ is a stacked Clifford torus of type $(N,k,\ell,m)$ with
$m\ge m_{0}$, decomposed as in \eqref{eq:decomp}, and $\cG=\cG[k,\ell,m]$. Since every element of
$\cG$ preserves each side of the tori, $\cG$ acts on functions on $\Sigma$ by composition, and we set
\[
H^{1}(\Sigma)^{\cG}:=\bigl\{u\in H^{1}(\Sigma):u\circ\mathfrak g=u\ \text{for all }\mathfrak g\in\cG\bigr\}.
\]
The group $\cG$ maps each toral region onto itself and permutes the tunnels of each layer among
themselves. Moreover every $\mathfrak g\in\cG$ carries the tunnel of the $i$-th layer over a removal
point $p$ onto the tunnel of the same layer over $\mathfrak gp$ by a map which, in the tunnel
coordinates \eqref{eq:kappa}, has the form $(t,\theta)\mapsto(t,\pm\theta+\theta_{0})$: indeed, by
the description of the action following \eqref{eq:G}, $\mathfrak g$ fixes the coordinate $\rz$ and
acts on $(\rx,\ry)$ by an affine isometry whose linear part is diagonal with entries $\pm1$, so that
it preserves the height $\rz^{K}_{i}+\tau_{i}t$, hence $t$, and rotates or reflects the polar angle
$\theta$ about the removal point. In particular a function which is constant on
each toral region and depends only on $t$ and on the layer in each tunnel is $\cG$-invariant. For
$u\in H^{1}(\Sigma)$ we write
\[
c_{j}:=\avg_{\cT[j]}u,\qquad d_{i}:=c_{i+1}-c_{i},\qquad
\cE_{\mathrm{tor}}:=\sum_{j=1}^{N}\int_{\cT[j]}|\nabla u|^{2},\qquad
\cE_{\mathrm{tun}}:=\sum_{i=1}^{N-1}\sum_{s=1}^{M}\int_{\cK_{i}^{(s)}}|\nabla u|^{2},
\]
all integrals being taken with respect to the area of $\Sigma$ unless a flat measure $d\rx\,d\ry$
or $dt\,d\theta$ is indicated. Finally we set
\begin{equation}\label{eq:weights}
A_{j}:=|\cT[j]|,\qquad \Theta_{j}:=M_{j}\,\pi R^{2},\qquad \widetilde A_{j}:=A_{j}+\Theta_{j},\qquad
\vartheta_{j}:=\frac{\Theta_{j}}{2\pi^{2}}=\frac{M_{j}}{2M}\cdot\frac{k}{100\pi\ell}\le\frac{1}{100\pi},
\end{equation}
where $M_{j}=\#Z_{j}\in\{M,2M\}$ is the number of tunnels adjoining $\cT[j]$; by \eqref{eq:Pj},
$|P_{j}|=2\pi^{2}(1-\vartheta_{j})$ and $\Theta_{j}=|\Tfl\setminus P_{j}|$.

\subsection{Areas}

\begin{lemma}\label{lem:mass}
For $m\ge m_{0}$, every tunnel half and every $j$,
\begin{gather*}
\bigl|\cK_{i}^{(s),\pm}\bigr|=\pi R^{2}\bigl(1+O(m^{-2})\bigr)\le2\pi R^{2},\qquad
A_{j}=|P_{j}|\bigl(1+O(m^{-2})\bigr),\\
\widetilde A_{j}=2\pi^{2}\bigl(1+O(m^{-2})\bigr),\qquad \pi^{2}\le A_{j}\le\widetilde A_{j}\le3\pi^{2}.
\end{gather*}
In particular, for each layer $i$, $M|\cK_{i}^{(s),\pm}|=\pi R^{2}M(1+O(m^{-2}))$: the halves of the
tunnels of one layer have, to leading order, exactly the area $\pi R^{2}M$ removed from each of the
two tori they join.
\end{lemma}

\begin{proof}
In the model metric $\tau_{i}^{2}\cosh^{2}t\,\widehat\chi$ the area of $\{0\le t\le a_{i}\}$ is
$2\pi\tau_{i}^{2}\int_{0}^{a_{i}}\cosh^{2}t\dd t=\pi\tau_{i}^{2}(a_{i}+\sinh a_{i}\cosh a_{i})
=\pi R^{2}(\tanh a_{i}+a_{i}\sech^{2}a_{i})$, using $\tau_{i}\cosh a_{i}=R$. Since
$a_{i}\ge cm^{2}$ by \eqref{eq:aiexp} and Corollary \ref{cor:rate}, and $1-\tanh a\le2e^{-2a}$,
$a\sech^{2}a\le4ae^{-2a}$, this is $\pi R^{2}(1+O(m^{-2}))$; by \eqref{eq:comp} and Lemma
\ref{lem:transfer} the area in $\Sigma$ differs from it by a factor $1+O(m^{-2})$. Likewise
$A_{j}=|P_{j}|(1+O(m^{-2}))$ by \eqref{eq:comp}. Since the discs $B(p,R)$, $p\in Z_{j}$, are
pairwise disjoint by \eqref{eq:ratio}, $|P_{j}|=2\pi^{2}-M_{j}\pi R^{2}$, so
$\widetilde A_{j}=2\pi^{2}+O(m^{-2})|P_{j}|=2\pi^{2}(1+O(m^{-2}))$. The bounds on $A_{j}$ follow
from $\vartheta_{j}\le1/(100\pi)$ for $m\ge m_{0}$.
\end{proof}

\subsection{Poincar\'e and trace inequalities on the toral regions}

\begin{lemma}[Uniform extension]\label{lem:extension}
Let $Z=Z_{j}$ and $P=P_{j}$ be as in \eqref{eq:Zj}--\eqref{eq:Pj}. There is a linear operator
$E:H^{1}(P)\to H^{1}(\Tfl)$ with $Eu|_{P}=u$, $E(1)=1$, which commutes with the action of $\cG$ and
satisfies
\[
\|\nabla Eu\|_{L^{2}(\Tfl)}\le\Lambda\,\|\nabla u\|_{L^{2}(P)}\qquad\text{for all }u\in H^{1}(P),
\]
with an absolute constant $\Lambda$.
\end{lemma}

\begin{proof}
For $p\in Z$ let $B_{p}:=B(p,R)$ and $\mathcal A_{p}:=B(p,2R)\setminus\overline{B_{p}}$. By
\eqref{eq:ratio} distinct points of $Z$ are at distance greater than $4R$, so the discs $B(p,2R)$ are
pairwise disjoint and each annulus $\mathcal A_{p}$ is contained in $P$. Let
$\iota_{p}(x):=p+R^{2}(x-p)/|x-p|^{2}$ be the inversion in the circle $\partial B_{p}$; it is a
conformal diffeomorphism of $\mathcal A_{p}$ onto $B_{p}\setminus\overline{B(p,R/2)}$ which is the
identity on $\partial B_{p}$, and its Jacobian determinant on $\mathcal A_{p}$ is $(R/|x-p|)^{4}\le1$.
Fix a smooth $\chi:[0,\infty)\to[0,1]$ with $\chi=0$ on $[0,\tfrac12]$, $\chi=1$ on $[\tfrac34,\infty)$
and $|\chi'|\le8$, and put $\chi_{p}(x):=\chi(|x-p|/R)$ and $\bar u_{p}:=\avg_{\mathcal A_{p}}u$
(with respect to $d\rx\,d\ry$). Define
\[
Eu:=u\ \text{on }P,\qquad
Eu:=\bar u_{p}+\chi_{p}\cdot\bigl((u-\bar u_{p})\circ\iota_{p}\bigr)\ \text{on }B_{p}\setminus\overline{B(p,R/2)},
\qquad Eu:=\bar u_{p}\ \text{on }\overline{B(p,R/2)}.
\]
The three definitions agree on the common boundaries, since $\iota_{p}$ is the identity and
$\chi_{p}=1$ on $\partial B_{p}$, and $\chi_{p}=0$ on $\partial B(p,R/2)$; hence $Eu\in H^{1}(\Tfl)$,
$E$ is linear, $Eu|_{P}=u$ and $E(1)=1$.

Let $w:=(u-\bar u_{p})\circ\iota_{p}$ on $B_{p}\setminus\overline{B(p,R/2)}$. By the conformal
invariance of the Dirichlet integral, $\int|\nabla w|^{2}=\int_{\mathcal A_{p}}|\nabla u|^{2}$; by the
change of variables $x=\iota_{p}(y)$ and the bound on the Jacobian, $\int w^{2}\le\int_{\mathcal A_{p}}(u-\bar u_{p})^{2}$;
and by the Poincar\'e inequality on the annulus $\{R<|x-p|<2R\}$, whose constant scales with
$R^{2}$, $\int_{\mathcal A_{p}}(u-\bar u_{p})^{2}\le C_{\mathcal A}R^{2}\int_{\mathcal A_{p}}|\nabla u|^{2}$
with $C_{\mathcal A}$ the Poincar\'e constant of $\{1<|x|<2\}$. Since $|\nabla\chi_{p}|\le8/R$,
\[
\int_{B_{p}}|\nabla Eu|^{2}\le2\int|\nabla w|^{2}+2\cdot\frac{64}{R^{2}}\int w^{2}
\le\bigl(2+128\,C_{\mathcal A}\bigr)\int_{\mathcal A_{p}}|\nabla u|^{2}.
\]
Summing over $p\in Z$, the annuli being pairwise disjoint subsets of $P$,
$\|\nabla Eu\|_{L^{2}(\Tfl)}^{2}\le(3+128\,C_{\mathcal A})\|\nabla u\|_{L^{2}(P)}^{2}$. Finally each
$\mathfrak g\in\cG$ acts on $\Tfl$ by an affine isometry with $\mathfrak g(Z)=Z$, and
$\mathfrak g(B_{p})=B_{\mathfrak gp}$, $\mathfrak g\circ\iota_{p}=\iota_{\mathfrak gp}\circ\mathfrak g$,
$\chi_{\mathfrak gp}\circ\mathfrak g=\chi_{p}$ and $\avg_{\mathcal A_{\mathfrak gp}}(u\circ\mathfrak g^{-1})
=\avg_{\mathcal A_{p}}u$; hence $E(u\circ\mathfrak g^{-1})=(Eu)\circ\mathfrak g^{-1}$.
\end{proof}

\begin{lemma}[Poincar\'e inequality on the perforated torus]\label{lem:poincare}
Let $P=P_{j}$ and $c_{0}:=\min(k,\ell)^{2}/\Lambda^{2}$. For every $w\in H^{1}(P)^{\cG}$,
\[
\int_{P}|\nabla w|^{2}\dd\rx\dd\ry\ \ge\ c_{0}\,m^{2}\int_{P}\bigl(w-\avg_{P}w\bigr)^{2}\dd\rx\dd\ry .
\]
\end{lemma}

\begin{proof}
Replacing $w$ by $w-\avg_{P}w$, which is still $\cG$-invariant and has the same gradient, we may
assume $\int_{P}w=0$. Let $W:=Ew\in H^{1}(\Tfl)^{\cG}$ and $\overline W:=\avg_{\Tfl}W$. Since $W=w$ on
$P$ and $\int_{P}w=0$, $|\Tfl|\,\overline W=\int_{\Tfl\setminus P}W$, so by the Cauchy--Schwarz
inequality $|\Tfl|^{2}\,\overline W^{2}\le|\Tfl\setminus P|\int_{\Tfl}W^{2}$, that is
$|\Tfl|\,\overline W^{2}\le\vartheta_{j}\int_{\Tfl}W^{2}$. The function $W-\overline W$ has vanishing
mean and is invariant under the translations by $(2X,0)$ and $(0,2Y)$, hence descends to the torus
$\bR^{2}/(2X\bZ\times2Y\bZ)$, whose Laplace eigenvalues are
$(\pi p/X)^{2}+(\pi q/Y)^{2}=2m^{2}(k^{2}p^{2}+\ell^{2}q^{2})$, $p,q\in\bZ$; the variational
characterisation of the first nonzero eigenvalue on that torus gives
\begin{align*}
\int_{\Tfl}|\nabla W|^{2}\ &\ge\ 2\min(k,\ell)^{2}m^{2}\int_{\Tfl}\bigl(W-\overline W\bigr)^{2}
=2\min(k,\ell)^{2}m^{2}\Bigl(\int_{\Tfl}W^{2}-|\Tfl|\,\overline W^{2}\Bigr)\\
&\ge\ 2\min(k,\ell)^{2}m^{2}(1-\vartheta_{j})\int_{\Tfl}W^{2}
\ \ge\ 2\min(k,\ell)^{2}m^{2}(1-\vartheta_{j})\int_{P}w^{2}.
\end{align*}
With Lemma \ref{lem:extension} and $2(1-\vartheta_{j})\ge1$ this gives
$\Lambda^{2}\int_{P}|\nabla w|^{2}\ge\min(k,\ell)^{2}m^{2}\int_{P}w^{2}$.
\end{proof}

\begin{corollary}\label{cor:poincare}
There is $c=c[k,\ell]>0$ such that for $m\ge m_{0}$, every $j$ and every $w\in H^{1}(\cT[j])^{\cG}$,
\[
\int_{\cT[j]}|\nabla w|^{2}\ \ge\ c\,m^{2}\int_{\cT[j]}\bigl(w-\avg_{\cT[j]}w\bigr)^{2}.
\]
\end{corollary}

\begin{proof}
Let $w'$ be $w$ transported to $P_{j}$ by the chart of \eqref{eq:decomp}; it is $\cG$-invariant.
The mean $\avg_{\cT[j]}w$ minimises $\int_{\cT[j]}(w-\kappa)^{2}$ over constants $\kappa$. By
\eqref{eq:comp} and Lemma \ref{lem:transfer}, $dA\le(1+\varepsilon_{m})\,d\rx\,d\ry$ on $\cT[j]$ and
$\int_{P_{j}}|\nabla w'|^{2}\dd\rx\dd\ry\le\frac{1+\varepsilon_{m}}{1-\varepsilon_{m}}\int_{\cT[j]}|\nabla w|^{2}$,
so with Lemma \ref{lem:poincare}
\begin{align*}
\int_{\cT[j]}\bigl(w-\avg_{\cT[j]}w\bigr)^{2}&\le\int_{\cT[j]}\bigl(w-\avg_{P_{j}}w'\bigr)^{2}
\le(1+\varepsilon_{m})\int_{P_{j}}\bigl(w'-\avg_{P_{j}}w'\bigr)^{2}\dd\rx\dd\ry\\
&\le\frac{1+\varepsilon_{m}}{c_{0}m^{2}}\int_{P_{j}}|\nabla w'|^{2}\dd\rx\dd\ry
\le\frac{(1+\varepsilon_{m})^{2}}{(1-\varepsilon_{m})\,c_{0}m^{2}}\int_{\cT[j]}|\nabla w|^{2}
\le\frac{4}{c_{0}m^{2}}\int_{\cT[j]}|\nabla w|^{2},
\end{align*}
the last step because $\varepsilon_{m}=3m^{-2}\le\frac13$ for $m\ge3$, whence
$(1+\varepsilon_{m})^{2}/(1-\varepsilon_{m})\le\frac{8}{3}$. This is the claim with $c=c_{0}/4$.
\end{proof}

\begin{lemma}[Aggregate trace estimate]\label{lem:trace}
There is $K=K[k,\ell]$ such that for $m\ge m_{0}$, every $j$ and every $u\in H^{1}(\cT[j])^{\cG}$,
\[
\sum_{\Gamma}\Bigl(\avg_{\Gamma}u-c_{j}\Bigr)^{2}\ \le\ K\int_{\cT[j]}|\nabla u|^{2},
\qquad c_{j}=\avg_{\cT[j]}u,
\]
the sum running over the $M_{j}$ end circles $\Gamma$ of the tunnels adjoining $\cT[j]$.
\end{lemma}

\begin{proof}
Let $u'$ be $u$ transported to $P_{j}$. For $p\in Z_{j}$ let $V_{p}$ be the Voronoi cell of $p$ with
respect to $Z_{j}$ in $\Tfl$. Since $Z_{j}$ is a lattice ($\mathsf L_{0}$, $\mathsf L_{1}$, or
$\mathsf L_{0}\cup\mathsf L_{1}$), the cells $V_{p}$ are translates of one convex polygon and tile $\Tfl$
with disjoint interiors; since distinct points of $Z_{j}$ are at distance at least $Y>2R$, the closed
disc $\overline{B(p,R)}$ lies in the interior of $V_{p}$. Thus $Q_{p}:=V_{p}\setminus B(p,R)$ is a
Lipschitz domain, all the $Q_{p}$ are translates of one domain $Q$, and after rescaling by $m$ the
domain $Q$ becomes one of two reference domains depending only on $k$ and $\ell$. On a bounded
Lipschitz domain the trace theorem and the Poincar\'e inequality give, for the circle
$\gamma=\partial B(p,R)$ with the mean taken with respect to the polar angle,
\[
\Bigl(\avg_{\gamma}v-\avg_{Q_{p}}v\Bigr)^{2}\le K_{0}\int_{Q_{p}}|\nabla v|^{2}\dd\rx\dd\ry
\qquad\text{for all }v\in H^{1}(Q_{p}),
\]
and both sides are invariant under dilations of the plane, so $K_{0}=K_{0}[k,\ell]$ is independent
of $m$. By Convention \ref{conv:avg}, $\avg_{\Gamma}u=\avg_{\gamma}u'$ for the end circle $\Gamma$
over $p$. Summing over $p\in Z_{j}$ and using \eqref{eq:comp} with Lemma \ref{lem:transfer},
\[
\sum_{p}\Bigl(\avg_{\gamma_{p}}u'-\avg_{Q_{p}}u'\Bigr)^{2}\le K_{0}\int_{P_{j}}|\nabla u'|^{2}\dd\rx\dd\ry
\le3K_{0}\int_{\cT[j]}|\nabla u|^{2}.
\]
On the other hand, by Jensen's inequality $|Q_{p}|\bigl(\avg_{Q_{p}}u'-c_{j}\bigr)^{2}\le\int_{Q_{p}}(u'-c_{j})^{2}$,
and $|Q_{p}|=|P_{j}|/M_{j}\ge\pi^{2}/(2k\ell m^{2})$, so by \eqref{eq:comp} and Corollary
\ref{cor:poincare}
\[
\sum_{p}\Bigl(\avg_{Q_{p}}u'-c_{j}\Bigr)^{2}\le\frac{2k\ell m^{2}}{\pi^{2}}\int_{P_{j}}(u'-c_{j})^{2}\dd\rx\dd\ry
\le\frac{3k\ell m^{2}}{\pi^{2}}\int_{\cT[j]}(u-c_{j})^{2}\le\frac{3k\ell}{\pi^{2}c}\int_{\cT[j]}|\nabla u|^{2}.
\]
Since $(x+y)^{2}\le2x^{2}+2y^{2}$, the lemma follows with $K=6K_{0}+6k\ell/(\pi^{2}c)$.
\end{proof}

\subsection{Mass on a channel}

The following lemma is the only place where the geometry of a tunnel enters beyond its conformal
length. It is stated for an abstract channel, since it will be used in that form in
Theorem \ref{thm:abstract}; the hypotheses are satisfied by the tunnels of a stacked Clifford torus
with $\rho=\tau_{i}\cosh t$ and $r=R$, as verified in the proof of Proposition \ref{prop:graph}.

\begin{lemma}[Mass on a channel]\label{lem:tunnelmass}
There is an absolute constant $C_{0}$ with the following property. Let $a\ge1$, $r>0$ and
$\varepsilon\in(0,\tfrac12]$, let $\rho:[-a,a]\to(0,\infty)$ satisfy
\begin{equation}\label{eq:rhodecay}
\rho(t)\ \le\ 2r\,e^{-(a-|t|)}\qquad(|t|\le a),
\end{equation}
and let $g$ be a Riemannian metric on $\bK_{a}$ with
$(1-\varepsilon)\rho^{2}\widehat\chi\le g\le(1+\varepsilon)\rho^{2}\widehat\chi$ and with
$|\cK^{\pm}|_{g}\le2\pi r^{2}$, where $\cK^{\pm}:=\{\pm t\ge0\}$ and $\Gamma_{\pm}:=\{t=\pm a\}$. Then
for every $u\in H^{1}(\bK_{a})$, writing
$\bar u_{\pm}:=\avg_{\Gamma_{\pm}}u$ and $\cE_{\pm}:=\int_{\cK^{\pm}}|\nabla u|^{2}$,
\[
\int_{\cK^{\pm}}\bigl(u-\bar u_{\pm}\bigr)^{2}\le C_{0}r^{2}\,\cE_{\pm},
\qquad
\Bigl|\int_{\cK^{\pm}}u^{2}-\bigl|\cK^{\pm}\bigr|\,\bar u_{\pm}^{2}\Bigr|
\le C_{0}r^{2}\Bigl(\cE_{\pm}+|\bar u_{\pm}|\,\cE_{\pm}^{1/2}\Bigr),
\]
all integrals and norms on the left being taken with respect to $g$. One may take $C_{0}=81$.
\end{lemma}

\begin{proof}
We treat $\cK^{+}=\{0\le t\le a\}$, the other half being identical. By
density we may take $u$ smooth. Put
$\cE_{0}:=\int_{\cK^{+}}(u_{t}^{2}+u_{\theta}^{2})\dd t\dd\theta$; by Lemma
\ref{lem:transfer}, the conformal invariance of the Dirichlet integral in dimension two and
$\tfrac{1+\varepsilon}{1-\varepsilon}\le3$, we have $\cE_{0}\le3\cE_{+}$, and
$dA_{g}\le(1+\varepsilon)\rho^{2}\dd t\dd\theta\le\tfrac32\rho^{2}\dd t\dd\theta$. By
\eqref{eq:rhodecay},
\begin{equation}\label{eq:expweight}
\rho^{2}\le4r^{2}e^{-2(a-t)}\qquad(0\le t\le a),
\qquad \int_{0}^{a}\rho^{2}\dd t\le2r^{2}.
\end{equation}
\emph{Interior.} For each $\theta$, $(u(t,\theta)-u(a,\theta))^{2}\le(a-t)\int_{t}^{a}u_{s}^{2}\dd s$,
so by \eqref{eq:expweight} and Fubini,
\[
\int_{0}^{a}\bigl(u(t,\theta)-u(a,\theta)\bigr)^{2}\rho^{2}\dd t
\le4r^{2}\int_{0}^{a}u_{s}(s,\theta)^{2}\Bigl(\int_{0}^{s}(a-t)e^{-2(a-t)}\dd t\Bigr)\dd s
\le r^{2}\int_{0}^{a}u_{s}(s,\theta)^{2}\dd s ,
\]
because $\int_{0}^{\infty}\sigma e^{-2\sigma}\dd\sigma=\tfrac14$.

\emph{Boundary.} Let $\mathcal A:=[a-1,a]\times\bS^{1}$ (recall $a\ge1$) and
$\bar u_{\mathcal A}:=\avg_{\mathcal A}u$ with respect to $dt\,d\theta$, and put $v:=u-\bar u_{\mathcal A}$.
For $t\in[a-1,a]$, $v(a,\theta)=v(t,\theta)+\int_{t}^{a}v_{s}\dd s$; integrating in $t$ over
$[a-1,a]$ and applying the Cauchy--Schwarz inequality,
$v(a,\theta)^{2}\le2\int_{a-1}^{a}v(t,\theta)^{2}\dd t+2\int_{a-1}^{a}v_{s}(s,\theta)^{2}\dd s$.
Integrating in $\theta$ and using the Poincar\'e inequality $\int_{\mathcal A}v^{2}\le\int_{\mathcal A}|\nabla v|^{2}$
on the flat cylinder $[0,1]\times\bS^{1}$, whose first nonzero Neumann eigenvalue is $1$, we obtain
$\int_{0}^{2\pi}v(a,\theta)^{2}\dd\theta\le4\int_{\mathcal A}|\nabla u|^{2}\dd t\dd\theta\le4\cE_{0}$.
Since $\bar u_{+}$ is the mean of $u(a,\cdot)$, $\int_{0}^{2\pi}(u(a,\theta)-\bar u_{+})^{2}\dd\theta
\le\int_{0}^{2\pi}v(a,\theta)^{2}\dd\theta\le4\cE_{0}$.

\emph{Conclusion.} By $(x+y)^{2}\le2x^{2}+2y^{2}$, the interior estimate integrated in $\theta$, and
the boundary estimate with \eqref{eq:expweight},
\[
\int_{\cK^{+}}(u-\bar u_{+})^{2}\rho^{2}\dd t\dd\theta
\le2r^{2}\cE_{0}+2\Bigl(\int_{0}^{a}\rho^{2}\dd t\Bigr)\int_{0}^{2\pi}\bigl(u(a,\theta)-\bar u_{+}\bigr)^{2}\dd\theta
\le18r^{2}\cE_{0},
\]
hence $\int_{\cK^{+}}(u-\bar u_{+})^{2}\dd A\le27r^{2}\cE_{0}\le81r^{2}\cE_{+}$. For the second
inequality, expand $u^{2}=\bar u_{+}^{2}+2\bar u_{+}(u-\bar u_{+})+(u-\bar u_{+})^{2}$ and estimate the
middle term by the Cauchy--Schwarz inequality, using
$|\cK^{+}|^{1/2}\le(2\pi r^{2})^{1/2}\le3r$:
$2|\bar u_{+}|\,|\cK^{+}|^{1/2}\bigl(81r^{2}\cE_{+}\bigr)^{1/2}\le54r^{2}|\bar u_{+}|\cE_{+}^{1/2}$. Both
inequalities hold with $C_{0}=81$.
\end{proof}

\subsection{The weighted graph model}

\begin{definition}\label{def:graph}
Let $\mathsf G=(V,E)$ be a finite connected graph without loops, with $|V|\ge2$, each edge $e\in E$ carrying an orientation with tail $e^{-}$ and head $e^{+}$. For weights $\widetilde A_{v}>0$ $(v\in V)$ and $C_{e}>0$ $(e\in E)$ set
\[
\mu(\widetilde A,C):=\min\left\{\frac{\sum_{e\in E}C_{e}\bigl(c_{e^{+}}-c_{e^{-}}\bigr)^{2}}{\sum_{v\in V}\widetilde A_{v}c_{v}^{2}}
\ :\ c\in\bR^{V}\setminus\{0\},\ \sum_{v\in V}\widetilde A_{v}c_{v}=0\right\},
\]
the lowest Rayleigh quotient of $\mathsf G$ with vertex weights $\widetilde A_{v}$ and edge
weights $C_{e}$ on functions orthogonal to the constants. The quotient does not depend on the choice
of orientations. The minimum exists, since the quotient is
homogeneous and continuous on the unit sphere of the hyperplane $\sum_{v}\widetilde A_{v}c_{v}=0$,
and it is positive: the numerator vanishes only on the constant vectors, $\mathsf G$ being connected,
and the only constant vector in that hyperplane is $0$, since $\sum_{v}\widetilde A_{v}>0$.

For $\mathsf G=P_{N}$ the path on $V=\{1,\dots,N\}$, with $e_{i}=\{i,i+1\}$ oriented from $i$ to
$i+1$, the numerator is $\sum_{i=1}^{N-1}C_{i}(c_{i}-c_{i+1})^{2}$; this is the case used in
Theorems A and B.
\end{definition}

\begin{lemma}[Evaluation of the model]\label{lem:mu}
Let $\widetilde A_{j}$ and $C_{i}$ be as in \eqref{eq:weights} and Definition \ref{def:conductance},
and let $C_{N}=8\pi^{2}/(2-b_{2}[N])$. Then, for $m\ge m_{0}$,
\[
\mu(\widetilde A,C)=\frac{C_{N}}{2\pi^{2}}\,\lambda_{1}\bigl(L(P_{N})\bigr)\bigl(1+O(m^{-2})\bigr)
=4\bigl(1+O(m^{-2})\bigr);
\]
more precisely, writing $\widetilde A_{j}=2\pi^{2}(1+\alpha_{j})$ and $C_{i}=C_{N}(1+\beta_{i})$
with $|\alpha_{j}|,|\beta_{i}|\le\eta$, one has $4(1-2\eta)\le\mu(\widetilde A,C)\le4(1+4\eta)$
whenever $\eta\le\tfrac1{16}$, so that in particular $\mu(\widetilde A,C)\le5$ for $m\ge m_{0}$.
\end{lemma}

\begin{proof}
By Proposition \ref{prop:b2}, $\lambda:=\lambda_{1}(L(P_{N}))=2-b_{2}[N]$, so
$C_{N}\lambda/(2\pi^{2})=4$. By Lemma \ref{lem:mass} and Proposition \ref{prop:conductance} the
numbers $\alpha_{j}$, $\beta_{i}$ satisfy $|\alpha_{j}|,|\beta_{i}|\le\eta:=Cm^{-2}$, and
$\eta\le\tfrac1{16}$ for $m\ge m_{0}$. Write $\|c\|$ for the Euclidean norm,
$\bar c:=N^{-1}\sum_{j}c_{j}$, and note $c^{\mathsf T}L(P_{N})c=\sum_{i}(c_{i}-c_{i+1})^{2}$.

\emph{Lower bound.} Let $c\ne0$ with $\sum_{j}\widetilde A_{j}c_{j}=0$. Then
$\sum_{j}c_{j}=-\sum_{j}\alpha_{j}c_{j}$, so $|\sum_{j}c_{j}|\le\eta\sqrt N\|c\|$ and
$N\bar c^{2}\le\eta^{2}\|c\|^{2}$. Since $c-\bar c\mathbf 1$ is orthogonal to the constants,
Lemma \ref{lem:pathspec}(ii) gives $c^{\mathsf T}Lc=(c-\bar c\mathbf1)^{\mathsf T}L(c-\bar c\mathbf1)
\ge\lambda\|c-\bar c\mathbf1\|^{2}=\lambda(\|c\|^{2}-N\bar c^{2})\ge\lambda(1-\eta^{2})\|c\|^{2}$.
Hence
\[
\frac{\sum_{i}C_{i}(c_{i}-c_{i+1})^{2}}{\sum_{j}\widetilde A_{j}c_{j}^{2}}
\ge\frac{C_{N}(1-\eta)\lambda(1-\eta^{2})}{2\pi^{2}(1+\eta)}=4\,\frac{(1-\eta)(1-\eta^{2})}{1+\eta}=4(1-\eta)^{2}\ge4(1-2\eta).
\]
\emph{Upper bound.} Let $c^{0}:=c^{(1)}$ be the eigenvector of Lemma \ref{lem:pathspec}(ii), which
satisfies $\sum_{j}c^{0}_{j}=0$ and $c^{0\mathsf T}Lc^{0}=\lambda\|c^{0}\|^{2}$, and put
$c:=c^{0}-\kappa\mathbf1$ with $\kappa:=\sum_{j}\widetilde A_{j}c^{0}_{j}/\sum_{j}\widetilde A_{j}$, so that
$\sum_{j}\widetilde A_{j}c_{j}=0$. Since $\sum_{j}\widetilde A_{j}c^{0}_{j}=2\pi^{2}\sum_{j}\alpha_{j}c^{0}_{j}$ and
$\sum_{j}\widetilde A_{j}\ge2\pi^{2}N(1-\eta)$, we have $|\kappa|\le\eta\sqrt N\|c^{0}\|/(N(1-\eta))\le2\eta\|c^{0}\|/\sqrt N$,
hence $\kappa^{2}\sum_{j}\widetilde A_{j}\le4\eta^{2}N^{-1}\|c^{0}\|^{2}\cdot2\pi^{2}N(1+\eta)\le12\pi^{2}\eta^{2}\|c^{0}\|^{2}$,
and the differences $c_{i}-c_{i+1}$ are those of $c^{0}$. Since
$\sum_{j}\widetilde A_{j}c_{j}^{2}=\sum_{j}\widetilde A_{j}(c^{0}_{j})^{2}-\kappa^{2}\sum_{j}\widetilde A_{j}
\ge2\pi^{2}(1-\eta-6\eta^{2})\|c^{0}\|^{2}\ge2\pi^{2}(1-2\eta)\|c^{0}\|^{2}$, we obtain
\[
\mu(\widetilde A,C)\le\frac{C_{N}(1+\eta)\lambda\|c^{0}\|^{2}}{2\pi^{2}(1-2\eta)\|c^{0}\|^{2}}
=4\,\frac{1+\eta}{1-2\eta}\le4(1+4\eta),
\]
the last inequality because $(1+\eta)\le(1+4\eta)(1-2\eta)$ for $\eta\le\tfrac18$. Finally
$4(1+4\eta)\le5$ for $\eta\le\tfrac1{16}$.
\end{proof}

\subsection{Graph decompositions}\label{ss:decomp}

We now isolate the properties of $\Sigma$ which the reduction uses. All of them have been established above for stacked Clifford tori, and the verification is carried out in the proof of Proposition \ref{prop:graph}.

\begin{definition}[$\cG$-graph decomposition]\label{def:decomp}
Let $(\Sigma,g)$ be a closed connected surface, $\cG$ a finite group of isometries of $(\Sigma,g)$, and $\mathsf G=(V,E)$ a finite connected graph without loops, with $|V|\ge2$, each edge $e\in E$ carrying an orientation with tail $e^{-}$ and head $e^{+}$. Suppose given integers $M_{e}\ge1$, real numbers
$a_{e}\ge1$ and $r_{e}>0$, functions $\rho_{e}:[-a_{e},a_{e}]\to(0,\infty)$, and constants
$\varepsilon\in(0,\tfrac12]$, $\eta\in[0,1]$, $\mathsf P>0$ and $\mathsf K>0$. A
\emph{$\cG$-graph decomposition of $\Sigma$ modelled on $\mathsf G$}, with these data, consists of
compact subsurfaces with boundary $\cB_{v}\subset\Sigma$ $(v\in V)$, the \emph{blocks}, and
$\cK_{e}^{(s)}\subset\Sigma$ $(e\in E$, $1\le s\le M_{e})$, the \emph{channels}, together with
diffeomorphisms $\mathfrak k_{e}^{(s)}:\bK_{a_{e}}\to\cK_{e}^{(s)}$, subject to {\rm(D1)--(D6)}
below. We transport the coordinates $(t,\theta)$ of $\bK_{a_{e}}=[-a_{e},a_{e}]\times\bS^{1}$ to
$\cK_{e}^{(s)}$ by $\mathfrak k_{e}^{(s)}$ and write
$\Gamma_{e,\pm}^{(s)}:=\mathfrak k_{e}^{(s)}(\{t=\pm a_{e}\})$ for the \emph{end circles},
$\cK_{e}^{(s),\pm}:=\mathfrak k_{e}^{(s)}(\{\pm t\ge0\})$ for the halves of a channel, and
$\avg_{\Gamma}u:=\frac1{2\pi}\int_{0}^{2\pi}u(\pm a_{e},\theta)\dd\theta$ for
$\Gamma=\Gamma_{e,\pm}^{(s)}$, in accordance with Convention \ref{conv:avg}.
\begin{enumerate}
\item[(D1)] \emph{Decomposition.} The blocks and the channels have pairwise disjoint interiors and
their union is $\Sigma$; distinct channels are disjoint; $\cK_{e}^{(s)}\cap\cB_{v}$ equals
$\Gamma_{e,-}^{(s)}$ if $v=e^{-}$, equals $\Gamma_{e,+}^{(s)}$ if $v=e^{+}$, and is empty otherwise;
and $\partial\cB_{v}$ is the disjoint union of the end circles of the channels of the edges incident
to $v$.
\item[(D2)] \emph{Channels.} For all $e$ and $s$,
\[
(1-\varepsilon)\,\rho_{e}^{2}\,\widehat\chi\ \le\ (\mathfrak k_{e}^{(s)})^{*}g\ \le\ (1+\varepsilon)\,\rho_{e}^{2}\,\widehat\chi
\quad\text{on }\bK_{a_{e}},\qquad
\rho_{e}(t)\le2r_{e}e^{-(a_{e}-|t|)}\quad(|t|\le a_{e}).
\]
\item[(D3)] \emph{Symmetry.} Every $\mathfrak g\in\cG$ maps each block $\cB_{v}$ onto itself and each
channel $\cK_{e}^{(s)}$ onto a channel $\cK_{e}^{(s')}$ of the same edge, and in the channel
coordinates $\mathfrak g$ acts by $(t,\theta)\mapsto(t,\pm\theta+\theta_{0})$ for some
$\theta_{0}\in\bR$ and some choice of sign.
\item[(D4)] \emph{Areas.} Writing $A_{v}:=|\cB_{v}|$, one has
$\bigl|\,|\cK_{e}^{(s),\pm}|-\pi r_{e}^{2}\bigr|\le\eta\,\pi r_{e}^{2}$ for all $e$, all $s$ and both
signs; and, setting $\Theta_{v}:=\sum_{e\ni v}M_{e}\pi r_{e}^{2}$, the sum being over the edges
incident to $v$, and $\widetilde A_{v}:=A_{v}+\Theta_{v}$, one has
$\Theta_{v}\le\tfrac12\widetilde A_{v}$.
\item[(D5)] \emph{Poincar\'e inequality on the blocks.} For every $v\in V$ and every $\cG$-invariant
$w\in H^{1}(\cB_{v})$,
$\int_{\cB_{v}}\bigl(w-\avg_{\cB_{v}}w\bigr)^{2}\le\mathsf P\int_{\cB_{v}}|\nabla w|^{2}$.
\item[(D6)] \emph{Aggregate trace inequality on the blocks.} For every $v\in V$ and every
$\cG$-invariant $u\in H^{1}(\cB_{v})$,
$\sum_{\Gamma}\bigl(\avg_{\Gamma}u-\avg_{\cB_{v}}u\bigr)^{2}\le\mathsf K\int_{\cB_{v}}|\nabla u|^{2}$,
the sum running over the end circles $\Gamma\subset\partial\cB_{v}$.
\end{enumerate}
We further set
\begin{equation}\label{eq:defXi}
C_{e}:=\frac{M_{e}\pi}{a_{e}},\qquad
\mathfrak a:=\sum_{e\in E}M_{e}r_{e}^{2},\qquad
\mathfrak b:=\sum_{e\in E}M_{e}^{1/2}r_{e}^{2},\qquad
\mathfrak c:=\sum_{e\in E}M_{e}^{-1/2},
\end{equation}
and we call
\begin{equation}\label{eq:Xi}
\Xi:=\varepsilon+\eta+\mathsf P+\mathfrak b+\mathfrak c
\end{equation}
the \emph{defect} of the decomposition.
\end{definition}

\begin{remark}\label{rem:decomp}
Two consequences of {\rm(D3)} will be used. First, a function which is constant on each block and
which, on the channels of each edge $e$, depends only on $t$, is $\cG$-invariant. Second, the
restriction to a block of a $\cG$-invariant function on $\Sigma$ is $\cG$-invariant, so that
{\rm(D5)} and {\rm(D6)} apply to it. Note that {\rm(D3)} does not require $\cG$ to act transitively
on the channels of an edge, and that no hypothesis is made on the metric of the blocks beyond
{\rm(D4)}--{\rm(D6)}.
\end{remark}

\subsection{The reduction theorem}\label{ss:reduction}

\begin{theorem}[Reduction of the invariant sector]\label{thm:abstract}
Let $(\Sigma,g)$ carry a $\cG$-graph decomposition modelled on $\mathsf G=(V,E)$ as in Definition
\ref{def:decomp}, with defect $\Xi$, and set
\[
a_{0}:=\min_{v\in V}\widetilde A_{v},\qquad a_{1}:=\max_{v\in V}\widetilde A_{v},\qquad
c_{1}:=\max_{e\in E}C_{e}.
\]
There are constants $C^{\ast}>0$ and $\Xi_{0}>0$, depending only on $|V|$, $|E|$, $a_{0}$, $a_{1}$,
$c_{1}$ and $\mathsf K$, such that if $\Xi\le\Xi_{0}$ then
\begin{equation}\label{eq:abstractmain}
\Bigl|\ \min\Bigl\{\cR(u):u\in H^{1}(\Sigma)^{\cG}\setminus\{0\},\ \textstyle\int_{\Sigma}u=0\Bigr\}
-\mu(\widetilde A,C)\ \Bigr|\ \le\ C^{\ast}\,\Xi .
\end{equation}
If moreover $\Sigma\subset\bS^{3}$ is a closed embedded minimal surface, $\cG<O(4)$ is generated by
reflections in great spheres, and $\mu(\widetilde A,C)>2+C^{\ast}\Xi$, then $\lambda_{1}(\Sigma)=2$.
Inspection of the proof shows that $C^{\ast}$ may be taken nondecreasing in $a_{1}$, $c_{1}$ and
$\mathsf K$ and nonincreasing in $a_{0}$, and $\Xi_{0}$ with the opposite monotonicity, so that in
applications it suffices to have a positive lower bound for the $\widetilde A_{v}$ and upper bounds
for the $\widetilde A_{v}$, the $C_{e}$ and $\mathsf K$.
\end{theorem}

\begin{proof}
Throughout the proof $C$ denotes a positive constant depending only on $|V|$, $|E|$, $a_{0}$,
$a_{1}$, $c_{1}$ and $\mathsf K$, which may change from line to line; for the duration of this proof
this supersedes Convention \ref{conv:notation}. We write $\mu:=\mu(\widetilde A,C)$,
$\cE_{\mathrm{blk}}:=\sum_{v}\int_{\cB_{v}}|\nabla u|^{2}$ and
$\cE_{\mathrm{ch}}:=\sum_{e}\sum_{s}\int_{\cK_{e}^{(s)}}|\nabla u|^{2}$ for the block and channel
parts of the Dirichlet energy of a function $u$.

\emph{Preliminaries.} Every edge is incident to exactly two vertices, so
$\sum_{v}\Theta_{v}=2\pi\mathfrak a$; with $\Theta_{v}\le\tfrac12\widetilde A_{v}$ this gives
\begin{equation}\label{eq:abprelim}
\mathfrak a\le\frac{1}{4\pi}\sum_{v}\widetilde A_{v}\le\frac{|V|a_{1}}{4\pi},\qquad
A_{v}\ge\tfrac12\widetilde A_{v}\ge\tfrac12a_{0},\qquad
|\Sigma|\ge\sum_{v}A_{v}\ge\tfrac12|V|a_{0},
\end{equation}
and, since $1\le M_{e}^{1/2}\le M_{e}$, also $\max_{e}r_{e}^{2}\le\mathfrak b\le\mathfrak a$. Next,
\begin{equation}\label{eq:mubar}
\mu\ \le\ \bar\mu:=\frac{8c_{1}|E|a_{1}^{2}}{a_{0}^{3}} .
\end{equation}
Indeed, fix $v_{0}\in V$, put $\kappa_{0}:=\widetilde A_{v_{0}}/\sum_{w}\widetilde A_{w}$ and
$c:=\mathbf 1_{v_{0}}-\kappa_{0}\mathbf 1$, so that $\sum_{v}\widetilde A_{v}c_{v}=0$ and $c\ne0$;
the numerator in Definition \ref{def:graph} is $\sum_{e\ni v_{0}}C_{e}\le2c_{1}|E|$, while, using
$1-\kappa_{0}=\sum_{w\ne v_{0}}\widetilde A_{w}\big/\sum_{w}\widetilde A_{w}\ge\frac{(|V|-1)a_{0}}{|V|a_{1}}\ge\frac{a_{0}}{2a_{1}}$,
the denominator is at least $\widetilde A_{v_{0}}(1-\kappa_{0})^{2}\ge a_{0}^{3}/(4a_{1}^{2})$.
Finally, by {\rm(D2)} and $|\cK_{e}^{(s),\pm}|\le(1+\eta)\pi r_{e}^{2}\le2\pi r_{e}^{2}$, Lemma
\ref{lem:tunnelmass} applies on every channel with $a=a_{e}$, $r=r_{e}$, $\rho=\rho_{e}$; and
Corollary \ref{cor:channel} applies on every channel and gives, for $u\in H^{1}(\Sigma)$,
\begin{equation}\label{eq:abchannel}
\int_{\cK_{e}^{(s)}}|\nabla u|^{2}\ \ge\ (1-2\varepsilon)\,\frac{\pi}{a_{e}}\,\delta_{s}^{2},
\qquad \delta_{s}:=\avg_{\Gamma_{e,+}^{(s)}}u-\avg_{\Gamma_{e,-}^{(s)}}u ,
\end{equation}
since $\frac{1-\varepsilon}{1+\varepsilon}\ge1-2\varepsilon$. We shall also use
$\frac{1+\varepsilon}{1-\varepsilon}\le1+4\varepsilon\le3$ for $\varepsilon\le\tfrac12$, and we
agree that $\Xi_{0}\le1$, so that $\Xi^{2}\le\Xi$ whenever $\Xi\le\Xi_{0}$.

\emph{Upper bound.} Let $c\in\bR^{V}$ realise $\mu$, normalised by
$\sum_{v}\widetilde A_{v}c_{v}^{2}=1$, so that $|c_{v}|\le\widetilde A_{v}^{-1/2}\le a_{0}^{-1/2}$;
put $d_{e}:=c_{e^{+}}-c_{e^{-}}$, so that $\sum_{e}C_{e}d_{e}^{2}=\mu$ and
$|d_{e}|\le2a_{0}^{-1/2}$. Define $u:=c_{v}$ on $\cB_{v}$ and, on each channel of the edge $e$ in
its coordinates $(t,\theta)$,
\[
u(t,\theta):=\frac{c_{e^{-}}+c_{e^{+}}}{2}+\frac{d_{e}}{2}\cdot\frac{t}{a_{e}} .
\]
Then $u=c_{e^{\pm}}$ on $\{t=\pm a_{e}\}$, so $u$ is continuous and piecewise smooth, hence in
$H^{1}(\Sigma)$; it is $\cG$-invariant by Remark \ref{rem:decomp}; $|u|\le a_{0}^{-1/2}$ and
$\avg_{\Gamma_{e,\pm}^{(s)}}u=c_{e^{\pm}}$.

On a channel of $e$ the Dirichlet integral of $u$ with respect to $\widehat\chi$ is
$(d_{e}/2a_{e})^{2}\cdot2a_{e}\cdot2\pi=\pi d_{e}^{2}/a_{e}$, so by {\rm(D2)}, Lemma
\ref{lem:transfer} and the conformal invariance of the Dirichlet integral in dimension two, the
energy $\cE_{s}$ of $u$ on that channel satisfies
$(1-2\varepsilon)\pi d_{e}^{2}/a_{e}\le\cE_{s}\le(1+4\varepsilon)\pi d_{e}^{2}/a_{e}\le3\pi d_{e}^{2}/a_{e}$,
whence, summing over $s$ and over $e$,
\begin{equation}\label{eq:abub1}
\Bigl|\int_{\Sigma}|\nabla u|^{2}-\mu\Bigr|\le4\varepsilon\sum_{e}C_{e}d_{e}^{2}=4\varepsilon\mu\le4\bar\mu\,\varepsilon,
\qquad \sum_{s}\cE_{s}\le3C_{e}d_{e}^{2}\le C .
\end{equation}
By Lemma \ref{lem:tunnelmass} with $\bar u_{\pm}=c_{e^{\pm}}$, then summing over $s$ and using
$\sum_{s}(\cE_{s}^{\pm})^{1/2}\le M_{e}^{1/2}(\sum_{s}\cE_{s})^{1/2}\le CM_{e}^{1/2}$ together with
{\rm(D4)},
\[
\Bigl|\sum_{s}\int_{\cK_{e}^{(s)}}u^{2}-M_{e}\pi r_{e}^{2}\bigl(c_{e^{-}}^{2}+c_{e^{+}}^{2}\bigr)\Bigr|
\ \le\ C\bigl(\eta M_{e}r_{e}^{2}+M_{e}^{1/2}r_{e}^{2}\bigr).
\]
Since $\int_{\cB_{v}}u^{2}=A_{v}c_{v}^{2}$ and
$\sum_{e}M_{e}\pi r_{e}^{2}(c_{e^{-}}^{2}+c_{e^{+}}^{2})=\sum_{v}\Theta_{v}c_{v}^{2}$, summing over
$e$ and using \eqref{eq:abprelim} gives
\begin{equation}\label{eq:abub2}
\Bigl|\int_{\Sigma}u^{2}-1\Bigr|=\Bigl|\int_{\Sigma}u^{2}-\sum_{v}\widetilde A_{v}c_{v}^{2}\Bigr|
\le C(\eta\mathfrak a+\mathfrak b)\le C\Xi .
\end{equation}
Likewise, by the first inequality of Lemma \ref{lem:tunnelmass} and the Cauchy--Schwarz inequality,
$\bigl|\int_{\cK_{e}^{(s),\pm}}(u-c_{e^{\pm}})\bigr|\le|\cK_{e}^{(s),\pm}|^{1/2}(C_{0}r_{e}^{2}\cE_{s}^{\pm})^{1/2}
\le3C_{0}^{1/2}r_{e}^{2}(\cE_{s}^{\pm})^{1/2}$, so that, summing over $s$ as before,
\[
\Bigl|\sum_{s}\int_{\cK_{e}^{(s)}}u-M_{e}\pi r_{e}^{2}\bigl(c_{e^{-}}+c_{e^{+}}\bigr)\Bigr|
\ \le\ C\bigl(\eta M_{e}r_{e}^{2}+M_{e}^{1/2}r_{e}^{2}\bigr);
\]
since $\int_{\cB_{v}}u=A_{v}c_{v}$ and
$\sum_{v}\widetilde A_{v}c_{v}=0$,
\begin{equation}\label{eq:abub3}
\Bigl|\int_{\Sigma}u\Bigr|=\Bigl|\int_{\Sigma}u-\sum_{v}\widetilde A_{v}c_{v}\Bigr|\le C(\eta\mathfrak a+\mathfrak b)\le C\Xi .
\end{equation}
Let $w:=u-|\Sigma|^{-1}\int_{\Sigma}u$. Then $w\in H^{1}(\Sigma)^{\cG}$, $\int_{\Sigma}w=0$,
$\int_{\Sigma}|\nabla w|^{2}=\int_{\Sigma}|\nabla u|^{2}$, and by \eqref{eq:abprelim},
\eqref{eq:abub2} and \eqref{eq:abub3},
$\int_{\Sigma}w^{2}=\int_{\Sigma}u^{2}-|\Sigma|^{-1}(\int_{\Sigma}u)^{2}\ge1-C\Xi$. Choosing
$\Xi_{0}$ small enough that $C\Xi_{0}\le\tfrac12$ we get $\int_{\Sigma}w^{2}\ge\tfrac12$, so
$w\ne0$, and with \eqref{eq:abub1},
\[
\min\Bigl\{\cR(u):u\in H^{1}(\Sigma)^{\cG}\setminus\{0\},\ \textstyle\int_{\Sigma}u=0\Bigr\}
\ \le\ \cR(w)\ \le\ \frac{\mu+4\bar\mu\varepsilon}{1-C\Xi}\ \le\ \mu+C\Xi .
\]

\emph{Lower bound.} Let $u\in H^{1}(\Sigma)^{\cG}$ with $\int_{\Sigma}u=0$ and
$\int_{\Sigma}u^{2}=1$, and put $\cE:=\int_{\Sigma}|\nabla u|^{2}=\cR(u)$. If $\cE>\bar\mu$ then
$\cE>\mu\ge\mu-C^{\ast}\Xi$ and there is nothing to prove; so assume $\cE\le\bar\mu$. Set
$c_{v}:=\avg_{\cB_{v}}u$ and $d_{e}:=c_{e^{+}}-c_{e^{-}}$. By Jensen's inequality
$\sum_{v}A_{v}c_{v}^{2}\le\sum_{v}\int_{\cB_{v}}u^{2}\le1$, so by \eqref{eq:abprelim}
$|c_{v}|\le A_{v}^{-1/2}\le(2/a_{0})^{1/2}$, $|d_{e}|\le2(2/a_{0})^{1/2}$ and
$\sum_{e}d_{e}^{2}\le8|E|/a_{0}$.

\emph{Step 1: energy.} For a channel $\cK_{e}^{(s)}$ set
$\delta_{\pm}^{(s)}:=\avg_{\Gamma_{e,\pm}^{(s)}}u-c_{e^{\pm}}$, so that $\delta_{s}$ of
\eqref{eq:abchannel} equals $d_{e}+\delta_{+}^{(s)}-\delta_{-}^{(s)}$. By {\rm(D6)}, applied to
$u|_{\cB_{e^{-}}}$ and to $u|_{\cB_{e^{+}}}$ and legitimate by Remark \ref{rem:decomp},
\begin{equation}\label{eq:abtrace}
\sum_{s}\bigl(\delta_{\pm}^{(s)}\bigr)^{2}\le\mathsf K\int_{\cB_{e^{\pm}}}|\nabla u|^{2}\le\mathsf K\cE_{\mathrm{blk}}
\le\mathsf K\bar\mu,\qquad\text{hence}\qquad
\sum_{s}\bigl(\delta_{+}^{(s)}-\delta_{-}^{(s)}\bigr)^{2}\le4\mathsf K\cE_{\mathrm{blk}} .
\end{equation}
By the Cauchy--Schwarz inequality and $2xy\le x^{2}+y^{2}$,
\[
M_{e}d_{e}^{2}-M_{e}^{1/2}\bigl(d_{e}^{2}+4\mathsf K\cE_{\mathrm{blk}}\bigr)\ \le\ \sum_{s=1}^{M_{e}}\delta_{s}^{2}
\ \le\ 2M_{e}d_{e}^{2}+8\mathsf K\cE_{\mathrm{blk}} ,
\]
so that, multiplying by $\pi/a_{e}=C_{e}/M_{e}$ and using \eqref{eq:abchannel},
\[
\sum_{s}\int_{\cK_{e}^{(s)}}|\nabla u|^{2}\ \ge\ C_{e}d_{e}^{2}
-\frac{C_{e}}{M_{e}^{1/2}}\bigl(d_{e}^{2}+4\mathsf K\cE_{\mathrm{blk}}\bigr)
-2\varepsilon\Bigl(2C_{e}d_{e}^{2}+\frac{8\mathsf KC_{e}\cE_{\mathrm{blk}}}{M_{e}}\Bigr).
\]
Summing over $e$ and using $\cE\ge\cE_{\mathrm{ch}}$, $\cE_{\mathrm{blk}}\le\bar\mu$,
$\sum_{e}d_{e}^{2}\le8|E|/a_{0}$, $C_{e}\le c_{1}$ and $M_{e}\ge1$,
\begin{equation}\label{eq:ablb1}
\cE\ \ge\ \sum_{e\in E}C_{e}d_{e}^{2}-C(\mathfrak c+\varepsilon)\ \ge\ \sum_{e\in E}C_{e}d_{e}^{2}-C\Xi .
\end{equation}

\emph{Step 2: mass.} By {\rm(D5)}, since $c_{v}$ is the mean of $u$ on $\cB_{v}$,
\[
\int_{\cB_{v}}u^{2}=A_{v}c_{v}^{2}+\int_{\cB_{v}}(u-c_{v})^{2}\le A_{v}c_{v}^{2}+\mathsf P\int_{\cB_{v}}|\nabla u|^{2}.
\]
On the channels of $e$, Lemma \ref{lem:tunnelmass} gives, with
$\bar u_{\pm}^{(s)}=c_{e^{\pm}}+\delta_{\pm}^{(s)}$ and $\cE_{s}$ the energy of $u$ on
$\cK_{e}^{(s)}$,
\[
\sum_{s}\int_{\cK_{e}^{(s)}}u^{2}\le\sum_{s}\Bigl(\bigl|\cK_{e}^{(s),-}\bigr|(\bar u_{-}^{(s)})^{2}
+\bigl|\cK_{e}^{(s),+}\bigr|(\bar u_{+}^{(s)})^{2}\Bigr)
+C_{0}r_{e}^{2}\sum_{s}\cE_{s}+C_{0}r_{e}^{2}\sum_{s}\bigl(|\bar u_{-}^{(s)}|+|\bar u_{+}^{(s)}|\bigr)\cE_{s}^{1/2}.
\]
By \eqref{eq:abtrace}, $\sum_{s}(\bar u_{\pm}^{(s)})^{2}\le2M_{e}c_{e^{\pm}}^{2}+2\mathsf K\bar\mu\le CM_{e}$,
whence by {\rm(D4)} and the Cauchy--Schwarz inequality
\[
\begin{aligned}
\sum_{s}\bigl|\cK_{e}^{(s),\pm}\bigr|(\bar u_{\pm}^{(s)})^{2}
&\le(1+\eta)\pi r_{e}^{2}\Bigl(M_{e}c_{e^{\pm}}^{2}
+2|c_{e^{\pm}}|M_{e}^{1/2}(\mathsf K\bar\mu)^{1/2}+\mathsf K\bar\mu\Bigr)\\
&\le M_{e}\pi r_{e}^{2}c_{e^{\pm}}^{2}+C\bigl(\eta M_{e}r_{e}^{2}+M_{e}^{1/2}r_{e}^{2}\bigr),
\end{aligned}
\]
while $C_{0}r_{e}^{2}\sum_{s}\cE_{s}\le C_{0}\bar\mu\,r_{e}^{2}$ and, again by Cauchy--Schwarz,
$C_{0}r_{e}^{2}\sum_{s}(|\bar u_{-}^{(s)}|+|\bar u_{+}^{(s)}|)\cE_{s}^{1/2}\le
C_{0}r_{e}^{2}(CM_{e})^{1/2}\bar\mu^{1/2}\le CM_{e}^{1/2}r_{e}^{2}$. Summing over $e$ and $v$ and
using $\sum_{e}M_{e}\pi r_{e}^{2}(c_{e^{-}}^{2}+c_{e^{+}}^{2})=\sum_{v}\Theta_{v}c_{v}^{2}$,
$\cE\le\bar\mu$ and \eqref{eq:abprelim},
\begin{equation}\label{eq:ablb2}
1=\int_{\Sigma}u^{2}\ \le\ \sum_{v}\widetilde A_{v}c_{v}^{2}+\mathsf P\bar\mu+C(\eta\mathfrak a+\mathfrak b)
\ \le\ \sum_{v}\widetilde A_{v}c_{v}^{2}+C\Xi .
\end{equation}

\emph{Step 3: mean.} By the first inequality of Lemma \ref{lem:tunnelmass} and Cauchy--Schwarz,
$\bigl|\int_{\cK_{e}^{(s),\pm}}(u-\bar u_{\pm}^{(s)})\bigr|\le3C_{0}^{1/2}r_{e}^{2}(\cE_{s}^{\pm})^{1/2}$,
so that $\sum_{s}\bigl|\int_{\cK_{e}^{(s),\pm}}(u-\bar u_{\pm}^{(s)})\bigr|\le
3C_{0}^{1/2}r_{e}^{2}M_{e}^{1/2}\bar\mu^{1/2}$; moreover, by {\rm(D4)}, \eqref{eq:abtrace} and
$\sum_{s}(\bar u_{\pm}^{(s)})^{2}\le CM_{e}$,
\[
\Bigl|\sum_{s}\bigl|\cK_{e}^{(s),\pm}\bigr|\bar u_{\pm}^{(s)}-M_{e}\pi r_{e}^{2}c_{e^{\pm}}\Bigr|
\le\pi r_{e}^{2}\Bigl|\sum_{s}\delta_{\pm}^{(s)}\Bigr|+\eta\pi r_{e}^{2}\sum_{s}\bigl|\bar u_{\pm}^{(s)}\bigr|
\le C\bigl(M_{e}^{1/2}r_{e}^{2}+\eta M_{e}r_{e}^{2}\bigr).
\]
Since $\int_{\cB_{v}}u=A_{v}c_{v}$, summing over $e$ and $v$ gives
\begin{equation}\label{eq:ablb3}
\Bigl|\sum_{v}\widetilde A_{v}c_{v}\Bigr|=\Bigl|\int_{\Sigma}u-\sum_{v}\widetilde A_{v}c_{v}\Bigr|
\le C(\eta\mathfrak a+\mathfrak b)\le C\Xi .
\end{equation}

\emph{Step 4: conclusion.} Put $\kappa:=\sum_{v}\widetilde A_{v}c_{v}\big/\sum_{v}\widetilde A_{v}$,
so that $|\kappa|\le C\Xi/(|V|a_{0})$ by \eqref{eq:ablb3}, and $c':=c-\kappa\mathbf 1$. Then
$\sum_{v}\widetilde A_{v}c'_{v}=0$, the differences $c'_{e^{+}}-c'_{e^{-}}=d_{e}$ are unchanged, and
by \eqref{eq:ablb2}
\[
\sum_{v}\widetilde A_{v}(c'_{v})^{2}=\sum_{v}\widetilde A_{v}c_{v}^{2}-\kappa^{2}\sum_{v}\widetilde A_{v}
\ \ge\ 1-C\Xi-C\Xi^{2}\ \ge\ 1-C\Xi .
\]
By Definition \ref{def:graph}, $\sum_{e}C_{e}d_{e}^{2}\ge\mu\sum_{v}\widetilde A_{v}(c'_{v})^{2}$,
an inequality which holds trivially if $c'=0$; hence
$\sum_{e}C_{e}d_{e}^{2}\ge\mu(1-C\Xi)\ge\mu-C\bar\mu\Xi$, which with \eqref{eq:ablb1} gives
$\cE\ge\mu-C\Xi$. Together with the upper bound this proves \eqref{eq:abstractmain}.

\emph{The last assertion.} By Lemma \ref{lem:takahashi}, $\lambda_{1}(\Sigma)\le2$. Suppose
$\lambda_{1}(\Sigma)<2$ and let $u$ be a first eigenfunction. By Lemma \ref{lem:CS}, $u$ is
$\cG$-invariant; since also $\int_{\Sigma}u=0$ and $u\ne0$, \eqref{eq:abstractmain} gives
$2>\lambda_{1}(\Sigma)=\cR(u)\ge\mu(\widetilde A,C)-C^{\ast}\Xi>2$, a contradiction. Hence
$\lambda_{1}(\Sigma)=2$.
\end{proof}

\subsection{Reduction of the invariant sector to the path graph}

\begin{proposition}\label{prop:graph}
For $m\ge m_{0}$,
\[
\Bigl|\ \min\Bigl\{\cR(u):u\in H^{1}(\Sigma)^{\cG}\setminus\{0\},\ \textstyle\int_{\Sigma}u=0\Bigr\}
-\mu(\widetilde A,C)\ \Bigr|\ \le\ \frac{C^{\ast}}{m}
\]
for a constant $C^{\ast}=C^{\ast}[N,k,\ell]$.
\end{proposition}

\begin{proof}
We verify the hypotheses of Definition \ref{def:decomp} for $\Sigma$ and then apply Theorem
\ref{thm:abstract}. We take for $\mathsf G$ the path $P_{N}$ on $V=\{1,\dots,N\}$, with the edge
$e_{i}=\{i,i+1\}$ oriented from $i$ to $i+1$, and we set
\begin{gather*}
\cB_{j}:=\cT[j],\qquad \cK_{e_{i}}^{(s)}:=\cK_{i}^{(s)},\qquad M_{e_{i}}:=M=k\ell m^{2},\\
a_{e_{i}}:=a_{i},\qquad r_{e_{i}}:=R,\qquad \rho_{e_{i}}(t):=\tau_{i}\cosh t,
\end{gather*}
with $\varepsilon:=\varepsilon_{m}=3m^{-2}$ and with the coordinates, end circles and halves of
\eqref{eq:decomp}.

\emph{{\rm(D1)}.} This is the decomposition \eqref{eq:decomp}: the pieces have pairwise disjoint
interiors, the tunnels are pairwise disjoint, and $\cK_{i}^{(s)}$ meets $\cT[i]$ exactly along
$\Gamma_{i,-}^{(s)}$ and $\cT[i+1]$ exactly along $\Gamma_{i,+}^{(s)}$. By Lemma
\ref{lem:initialtor} and \eqref{eq:Pj}, $\partial\cT[j]$ is the disjoint union of the $M_{j}$ circles
lying over $\partial B(p,R)$, $p\in Z_{j}$, that is of the end circles of the tunnels adjoining
$\cT[j]$.

\emph{{\rm(D2)}.} The two-sided bound is the first line of \eqref{eq:comp}, and
$\varepsilon_{m}\le\tfrac12$ for $m\ge m_{0}$. Since $\tau_{i}\cosh a_{i}=R$ by \eqref{eq:ai} we
have $\tau_{i}\le2Re^{-a_{i}}$, whence
$\rho_{e_{i}}(t)=\tau_{i}\cosh t\le\tau_{i}e^{|t|}\le2Re^{-(a_{i}-|t|)}$. Finally $a_{i}\ge1$ for
$m\ge m_{0}$, by \eqref{eq:aiexp} and Corollary \ref{cor:rate}.

\emph{{\rm(D3)}} is the statement established at the beginning of this section.

\emph{{\rm(D4)}.} By Lemma \ref{lem:mass}, $|\cK_{i}^{(s),\pm}|=\pi R^{2}(1+O(m^{-2}))$, so
$\eta=Cm^{-2}\le1$ for $m\ge m_{0}$. The vertex $j$ is incident to one edge if $j\in\{1,N\}$ and to
two otherwise, so $\sum_{e\ni j}M_{e}\pi R^{2}=M_{j}\pi R^{2}$, which is the quantity $\Theta_{j}$
of \eqref{eq:weights}; hence $A_{j}$, $\Theta_{j}$ and $\widetilde A_{j}$ are those of
\eqref{eq:weights}. By \eqref{eq:weights} and Lemma \ref{lem:mass},
$\Theta_{j}=2\pi^{2}\vartheta_{j}\le\pi/50\le\tfrac12\pi^{2}\le\tfrac12\widetilde A_{j}$.

\emph{{\rm(D5)}} is Corollary \ref{cor:poincare}, with $\mathsf P=1/(cm^{2})$ and $c=c[k,\ell]$.

\emph{{\rm(D6)}} is Lemma \ref{lem:trace}, with $\mathsf K=K[k,\ell]$.

The conductance $C_{e_{i}}=M\pi/a_{i}$ is the $C_{i}$ of Definition \ref{def:conductance}, and the
weighted graph of Definition \ref{def:graph} is the weighted path occurring in Lemma \ref{lem:mu},
so that $\mu(\widetilde A,C)$ is the same quantity in both statements. By Lemma \ref{lem:mass} we
have $\pi^{2}\le a_{0}\le a_{1}\le3\pi^{2}$, and $c_{1}\le C[N]$ by Proposition
\ref{prop:conductance}; these bounds, together with $|V|=N$, $|E|=N-1$ and $\mathsf K=K[k,\ell]$,
depend only on $N$, $k$ and $\ell$, so the constants $C^{\ast}$ and $\Xi_{0}$ of Theorem
\ref{thm:abstract} do as well. Finally $\varepsilon$, $\eta$ and $\mathsf P$ are $O(m^{-2})$, while,
by \eqref{eq:XY},
\[
\mathfrak b=(N-1)M^{1/2}R^{2}=\frac{(N-1)\sqrt{k\ell}}{100\,\ell^{2}m},\qquad
\mathfrak c=(N-1)M^{-1/2}=\frac{N-1}{\sqrt{k\ell}\,m},
\]
so that $\Xi=O(m^{-1})$ with a constant depending only on $N,k,\ell$. Enlarging $m_{0}$ so that
$\Xi\le\Xi_{0}$, Theorem \ref{thm:abstract} gives the assertion.
\end{proof}

\begin{remark}
The weights $\widetilde A_{j}=A_{j}+\Theta_{j}$ are the areas of the tori \emph{before} the discs
are removed. Steps 2 and 3 in the proof of Theorem \ref{thm:abstract}, and the corresponding
computations in its upper bound, show that, to the order $O(m^{-1})$, the halves of the tunnels
restore exactly the removed discs, both in the $L^{2}$ norm and in the mean; no term of the size
$\vartheta_{j}$ of the removed area survives. This is what makes the reduction two-sided with error
$O(m^{-1})$ and is used in Theorem B; for Theorem A alone a one-sided bound with a loss of order
$\vartheta_{j}$ would suffice.
\end{remark}

\section{Proofs of Theorems A and B}\label{sec:proof}

\begin{proof}[Proof of Theorem B]
Part (i) is Proposition \ref{prop:b2}. Part (ii): the first identity is Proposition
\ref{prop:conductance}; the identity $C_{N}\lambda_{1}(L(P_{N}))/(2\pi^{2})=4$ is the computation
\[
\frac{C_{N}}{2\pi^{2}}\,\lambda_{1}\bigl(L(P_{N})\bigr)
=\frac{1}{2\pi^{2}}\cdot\frac{8\pi^{2}}{2-b_{2}[N]}\cdot\bigl(2-b_{2}[N]\bigr)=4 ,
\]
in which $\lambda_{1}(L(P_{N}))=2-b_{2}[N]$ is Lemma \ref{lem:incidence} together with the
identification $b_{2}[N]=\lambda_{\max}(A(P_{N-1}))$ of Proposition \ref{prop:b2}; the evaluation
$\mu(\widetilde A,C)=4+O(m^{-2})$ is Lemma \ref{lem:mu}. Part (iii) follows from
$\lambda_{1}(L(P_{N}))=2(1-\cos(\pi/N))=\pi^{2}N^{-2}-\tfrac{\pi^{4}}{12}N^{-4}+O(N^{-6})$ and from
Proposition \ref{prop:conductance}.
\end{proof}

\begin{proof}[Proof of Theorem A]
By Lemma \ref{lem:takahashi}, $\lambda_{1}(\Sigma)\le2$. Suppose $\lambda_{1}(\Sigma)<2$ and let $u$
be a first eigenfunction. The surface $\Sigma$ is closed, embedded, minimal and $\cG$-invariant, and
$\cG$ is generated by reflections in great spheres by Lemma \ref{lem:reflgen}; hence Lemma
\ref{lem:CS} gives $u\in H^{1}(\Sigma)^{\cG}$, while $\int_{\Sigma}u=0$ and
$\cR(u)=\lambda_{1}(\Sigma)<2$. On the other hand, by Proposition \ref{prop:graph} and Lemma
\ref{lem:mu}, for $m\ge m_{0}[N,k,\ell]$ every $\cG$-invariant function with vanishing mean satisfies
\[
\cR(u)\ \ge\ \mu(\widetilde A,C)-\frac{C^{\ast}}{m}\ \ge\ 4\bigl(1-Cm^{-2}\bigr)-\frac{C^{\ast}}{m}\ \ge\ 3,
\]
a contradiction. Hence $\lambda_{1}(\Sigma)=2$.

For the second assertion, note that the Laplacian of $\Sigma$ commutes with the isometric action of
$\cG$, so the closed subspace $L^{2}(\Sigma)^{\cG}$ of $\cG$-invariant functions is spanned by
$\cG$-invariant eigenfunctions, and the smallest nonzero eigenvalue occurring there is
$\min\{\cR(u):u\in H^{1}(\Sigma)^{\cG}\setminus\{0\},\ \int_{\Sigma}u=0\}$ by the variational
characterisation applied in that subspace, the kernel being spanned by the constants. By
Proposition \ref{prop:graph} and Lemma \ref{lem:mu} this minimum equals
$\mu(\widetilde A,C)+O(m^{-1})=4+O(m^{-1})$.

Finally, by Corollary \ref{cor:wiygulW} every surface produced by \cite[Theorem 6.50]{Wiygul} with $m\ge m_{0}$ is a stacked Clifford torus of type $(N,k,\ell,m)$, after $m_{0}$ has been enlarged to exceed the threshold of that corollary; hence Yau's conjecture holds for all these surfaces.
\end{proof}

\begin{remark}\label{rem:uniform}
The proof uses about the symmetry group only that it is generated by reflections in great spheres,
that it preserves each torus, and that it acts on functions without sign. Whether or not
$\Sigma$ admits further isometries exchanging the two sides of $\bT$, which for $k\ne\ell$ it does
not \cite[Remark 1.2]{Wiygul} and for $k=\ell$ is not decided in \cite{Wiygul}, plays no role; in particular the argument is the same for all $N\ge2$. The restrictions of the coordinate functions of $\bR^{4}$ are eigenfunctions with eigenvalue $2$ and are not $\cG$-invariant, in accordance with the value $4+O(m^{-1})$ of the invariant sector.
\end{remark}

\section{The balancing--spectrum identity}\label{sec:remarks}

\subsection{The mechanism}
Consider a chain of $N$ tori of area $2\pi^{2}$ joined by families of tunnels of total conductance
$C$. The mode which is constant on each torus has Rayleigh quotient
$\frac{C}{2\pi^{2}}\,\lambda_{1}(L(P_{N}))\sim\frac{C}{2\pi^{2}}\,\pi^{2}N^{-2}$. Were $C$ bounded
independently of $N$, this would tend to $0$ and, by Lemma \ref{lem:takahashi}, long chains would
violate Yau's conjecture. What excludes this is the balancing. The linearised balancing system says
that the vector of waist ratios is a positive eigenvector of $A(P_{N-1})$ with eigenvalue $b_{2}$,
hence its Perron vector (Lemma \ref{lem:perron}, Proposition \ref{prop:b2}); the incidence identity
(Lemma \ref{lem:incidence}) makes the spectral gap of the chain equal to $2-b_{2}$; and the
conformal length \eqref{eq:aiexp} of the tunnels is proportional to $2-b_{2}$, so their conductance is
proportional to $(2-b_{2})^{-1}$. The cancellation is exact for every $N$ and requires no evaluation of $b_{2}$. Equivalently, at the level of the limiting formulas, increasing $N$ decreases $2-b_{2}[N]$ and hence increases the basic waist scale: in this sense the necks of a longer chain
are forced to be fatter, by exactly the amount responsible for the cancellation. Remark \ref{rem:fiedler} adds that the
vector of heights of the tori is the Fiedler vector of $P_{N}$, which is why the mode governed by the gap, and not another block-constant mode, is the one realised by the configuration.

\subsection{The constant 4}

\begin{remark}\label{rem:jacobi}
The surviving constant $4$ is the zeroth-order coefficient of the Jacobi
operator \eqref{eq:jacobi}. Its geometric origin is already visible in the
linearised balancing calculation: with our sign conventions the parallel
torus $T_z$ has mean curvature
\[
2\tan(2z)=4z+O(z^3),
\]
and this linear term is balanced against the catenoidal fluxes. Thus the
local origin of the constant $4$ is the Jacobi operator, while its appearance
in the limiting spectral model also uses the balancing conditions.

For comparison, the doubling analysis uses the even--odd decomposition
associated with a side-exchanging involution, exactly or after a small
perturbation of the metric; the neck-controlled two-sheet mode lies in the
odd sector. For doublings of the equatorial two-sphere, the analogous Jacobi
coefficient is
\[
|A|^{2}+\Ric(\nu,\nu)=0+2=2,
\]
and the ambient coordinate function vanishing on the base sphere restricts
to an exact eigenfunction with eigenvalue $2$.

What Sections \ref{sec:balancing}--\ref{sec:sector} add is that the limiting
value $4$ is obtained, for every fixed $N$, by the mode governed by the
\emph{spectral gap} of the chain, which is the statement needed for
Theorem A: a mode with quotient close to $4$ would not by itself exclude
lower block-constant modes in a long chain.
\end{remark}

\subsection{Trees}
The identity $\lambda_{1}(L(P_{N}))=2-\lambda_{\max}(A(P_{N-1}))$ is a property of paths. We record
the general statement.

\begin{proposition}\label{prop:trees}
Let $T$ be a finite tree with at least one edge, $L(T)$ its combinatorial Laplacian and $\Line(T)$
its line graph. Then
\[
\lambda_{1}\bigl(L(T)\bigr)=2+\lambda_{\min}\bigl(A(\Line(T))\bigr),\qquad
\lambda_{\max}\bigl(L(T)\bigr)=2+\lambda_{\max}\bigl(A(\Line(T))\bigr),
\]
and $\lambda_{1}(L(T))=2-\lambda_{\max}(A(\Line(T)))$ holds if and only if $T$ is a path.
\end{proposition}

\begin{proof}
Let $|D|$ be the unsigned $|V|\times|E|$ incidence matrix of $T$. Then $|D||D|^{\mathsf T}=\operatorname{diag}(\deg)+A(T)=:Q(T)$
is the signless Laplacian and $|D|^{\mathsf T}|D|=2I+A(\Line(T))$, by the same entry computation as in
Lemma \ref{lem:incidence} but without signs. Since $T$ is bipartite, with parts $V_{1}\sqcup V_{2}$,
the diagonal matrix $S$ with entries $+1$ on $V_{1}$ and $-1$ on $V_{2}$ satisfies $SA(T)S=-A(T)$,
because every edge joins $V_{1}$ to $V_{2}$; hence $SQ(T)S=L(T)$ and $Q(T)$ and $L(T)$ have the same
spectrum (see \cite{BrouwerHaemers} for these standard facts). As $T$ is connected, $\ker L(T)$ is
one-dimensional, so $Q(T)$ has exactly $|V|-1=|E|$ nonzero eigenvalues counted with multiplicity;
these are the nonzero eigenvalues of $|D|^{\mathsf T}|D|$, an $|E|\times|E|$ matrix, which therefore
has no zero eigenvalue. Thus the spectrum of $2I+A(\Line(T))$ is the spectrum of $L(T)$ with $0$
removed, which gives the two identities.

For the last statement, the identity in question holds if and only if
$\lambda_{\min}(A(\Line(T)))=-\lambda_{\max}(A(\Line(T)))$. We use the following classical fact: for a
connected graph $H$ with adjacency matrix $A$ and spectral radius $\rho$, the number $-\rho$ is an
eigenvalue of $A$ if and only if $H$ is bipartite. Indeed, if $Av=-\rho v$ with $v\ne0$, then
$\rho|v|=|Av|\le A|v|$ entrywise, so $\langle A|v|,|v|\rangle\ge\rho\|v\|^{2}$ and $|v|$ maximises the
Rayleigh quotient of $A$; hence $A|v|=\rho|v|$, $|v|$ is the Perron vector, which is positive since
$A$ is irreducible, and equality $|Av|=A|v|$ forces, for every vertex $i$, all $v_{j}$ with $j\sim i$
to have the common sign $-\operatorname{sgn}v_{i}$. Thus $\operatorname{sgn}v$ is a proper
$2$-colouring of $H$. Conversely, if $H$ is bipartite with Perron vector $w$ and $S$ is the sign
matrix of the bipartition, then $A(Sw)=-SAw=-\rho\,Sw$. Now the line graph of a connected graph is
connected. If some vertex of $T$ has degree at least $3$, three edges at that vertex are pairwise
adjacent and form a triangle in $\Line(T)$, which is then not bipartite, and the identity fails. If
every vertex of $T$ has degree at most $2$, then $T$, being a connected tree, is a path $P_{N}$, and
$\Line(P_{N})=P_{N-1}$ is bipartite, so the identity holds. Hence the identity holds exactly when $T$
is a path.
\end{proof}

The algebraic content of the mechanism of Theorem B is the following statement about the operator
$D^{\mathsf T}D$ of a tree. As the proof records, $L(T)=D_{\sigma}D_{\sigma}^{\mathsf T}$ for every
orientation $\sigma$; we write $\lambda_{1}(L(T))$ and $\lambda_{\max}(L(T))$ for its smallest
nonzero and its largest eigenvalue.

\begin{proposition}[The balancing operator of a tree]\label{prop:balancing}
Let $T$ be a finite tree with $n\ge2$ vertices and edge set $E$, let $\sigma$ be an orientation of
its edges, let $D_{\sigma}$ be the corresponding incidence matrix as in Section \ref{sec:balancing},
and set $A_{\sigma}:=2I-D_{\sigma}^{\mathsf T}D_{\sigma}$, so that $(A_{\sigma})_{ee}=0$, while for
$e\ne f$ one has $(A_{\sigma})_{ef}=0$ if $e$ and $f$ are not adjacent and, if they meet at a vertex
$v$, $(A_{\sigma})_{ef}=+1$ if $v$ is the head of one of them and the tail of the other, and
$(A_{\sigma})_{ef}=-1$ otherwise. Then:
\begin{enumerate}
\item[(i)] $D_{\sigma}^{\mathsf T}D_{\sigma}$ is positive definite and its spectrum, with
multiplicities, is the spectrum of $L(T)$ with the eigenvalue $0$ removed; in particular every
eigenvalue $\kappa$ of $D_{\sigma}^{\mathsf T}D_{\sigma}$ satisfies $\kappa\ge\lambda_{1}(L(T))$.
\item[(ii)] $A_{\sigma}$ has nonnegative entries if and only if $T$ is a path and $\sigma$ orients it
consistently, and in that case $A_{\sigma}=A(\Line(T))$.
\item[(iii)] If $A_{\sigma}$ has nonnegative entries and some $\tau\in\bR^{E}$ with positive entries
satisfies $D_{\sigma}^{\mathsf T}D_{\sigma}\tau=\kappa\tau$, then $\kappa=\lambda_{1}(L(T))$ and
$\tau$ is a positive multiple of the Perron vector of $A(\Line(T))$.
\item[(iv)] For the star $T=K_{1,3}$ with all three edges oriented away from the centre one has
$D_{\sigma}^{\mathsf T}D_{\sigma}=2I+A(K_{3})$, whose spectrum is $\{4,1,1\}$; its eigenvectors with
positive entries are exactly the positive multiples of $(1,1,1)$, and the corresponding eigenvalue is
$\kappa=4=\lambda_{\max}(L(K_{1,3}))$.
\end{enumerate}
\end{proposition}

\begin{proof}
The entries of $D_{\sigma}^{\mathsf T}D_{\sigma}$ are computed as in Lemma \ref{lem:incidence}:
$(D_{\sigma}^{\mathsf T}D_{\sigma})_{ee}=2$, while for $e\ne f$ meeting at $v$ the only nonzero term of
$\sum_{w}D_{we}D_{wf}$ is the one at $w=v$, which equals $-1$ if $v$ is the head of one of $e,f$ and
the tail of the other and $+1$ if it is the head of both or the tail of both; this gives the stated
form of $A_{\sigma}$. Likewise $(D_{\sigma}D_{\sigma}^{\mathsf T})_{vw}$ equals the degree of $v$ if
$v=w$, equals $(+1)(-1)=-1$ if $\{v,w\}$ is an edge and vanishes otherwise, whatever the orientation;
so $D_{\sigma}D_{\sigma}^{\mathsf T}=L(T)$.

(i) Since $T$ is connected, $\ker D_{\sigma}^{\mathsf T}$ consists of the constant vectors, so
$\operatorname{rank}D_{\sigma}=n-1=|E|$ and the $|E|\times|E|$ matrix
$D_{\sigma}^{\mathsf T}D_{\sigma}$ is positive definite. The nonzero eigenvalues of
$D_{\sigma}D_{\sigma}^{\mathsf T}$ and $D_{\sigma}^{\mathsf T}D_{\sigma}$ coincide with
multiplicities, and $L(T)$ has exactly $n-1=|E|$ of them; these are therefore all the eigenvalues of
$D_{\sigma}^{\mathsf T}D_{\sigma}$.

(ii) Suppose $A_{\sigma}\ge0$ entrywise. If some vertex $v$ had degree at least $3$, then among three
edges at $v$ two would be both incoming or both outgoing at $v$, and the corresponding entry of
$A_{\sigma}$ would be $-1$. Hence every vertex of $T$ has degree at most $2$, and $T$, being a
connected tree, is a path; and each vertex of degree $2$ must be the head of one of the two edges
meeting there and the tail of the other, so $\sigma$ orients the path consistently. Conversely, for a consistently oriented path every pair of
adjacent edges meets at a vertex which is the head of one and the tail of the other, so
$(A_{\sigma})_{ef}=1$ exactly when $e$ and $f$ are adjacent, that is $A_{\sigma}=A(\Line(T))$.

(iii) By (ii), $A_{\sigma}=A(\Line(T))$ with $\Line(T)=P_{n-1}$, which is connected, so $A_{\sigma}$
is nonnegative and irreducible. From $D_{\sigma}^{\mathsf T}D_{\sigma}\tau=\kappa\tau$ we get
$A_{\sigma}\tau=(2-\kappa)\tau$ with $\tau>0$, so by the Perron--Frobenius theorem $\tau$ is, up to a
positive multiple, the Perron vector and $2-\kappa=\lambda_{\max}(A_{\sigma})$. Hence $\kappa$ is the
smallest eigenvalue of $D_{\sigma}^{\mathsf T}D_{\sigma}$, which is $\lambda_{1}(L(T))$ by (i).

(iv) The centre is the tail of all three edges, so $(A_{\sigma})_{ef}=-1$ for every pair of distinct
edges, that is $A_{\sigma}=-A(K_{3})$ and $D_{\sigma}^{\mathsf T}D_{\sigma}=2I+A(K_{3})$. The
spectrum of $A(K_{3})$ is $\{2,-1,-1\}$, whence the spectrum $\{4,1,1\}$; the eigenspace for $4$ is
spanned by $(1,1,1)$, while the eigenspace for $1$ consists of the vectors with vanishing sum, none
of which has positive entries. Finally the spectrum of $L(K_{1,3})$ is $\{0,1,1,4\}$, so
$\lambda_{\max}(L(K_{1,3}))=4$.
\end{proof}

\begin{remark}\label{rem:linearmodel}
Proposition \ref{prop:trees} delimits the reach of the mechanism of Theorem B, and Proposition
\ref{prop:balancing} isolates its algebraic core. In the linear model
which produced \eqref{eq:system}, blocks $v$ of a base surface with Jacobi constant $q$ and area
$|\Sigma_{0}|$ carry heights $h_{v}$, necks $e$ carry waists $\tau_{e}$ and a common leading
conformal length $a$, and the two balancing relations read
$q|\Sigma_{0}|h=2\pi M\,D_{\sigma}\tau$ and $D_{\sigma}^{\mathsf T}h=2a\tau$, where $D_{\sigma}$ is the
incidence matrix of $T$ with each edge oriented from the lower to the upper block, as in Lemma
\ref{lem:incidence} for the chain. Hence
$D_{\sigma}^{\mathsf T}D_{\sigma}\tau=\kappa\tau$ with $\kappa=q|\Sigma_{0}|a/(\pi M)$, the total
conductance is $C=M\pi/a=q|\Sigma_{0}|/\kappa$, and the lowest block-constant Rayleigh quotient is
$q\,\lambda_{1}(L(T))/\kappa$. Now Proposition \ref{prop:balancing}(ii)--(iii) applies: the matrix
$A_{\sigma}=2I-D_{\sigma}^{\mathsf T}D_{\sigma}$ has nonnegative entries exactly when every block has
at most one neighbour on each side, that is when $T$ is a path traversed consistently; then the
positivity of $\tau$ forces $\kappa=\lambda_{1}(L(T))$, and the quotient is $q$. For a
star with three blocks on the same side of the centre one finds instead, by Proposition
\ref{prop:balancing}(iv), $\kappa=4=\lambda_{\max}(L(K_{1,3}))$ and the quotient
$q\lambda_{1}/\lambda_{\max}=q/4$, which for the Clifford torus is $1<2$; such a configuration is of
course not realisable by parallel tori. The exact cancellation is therefore an intrinsically
one-dimensional phenomenon.
\end{remark}

\begin{problem}\label{prob:graphs}
Let a closed embedded minimal surface be produced by gluing congruent blocks of a base surface
$\Sigma_{0}$ with constant $q=|A|^{2}+\Ric(\nu,\nu)$ along the edges of a finite graph $G$, the waist
parameters being determined by linearised balancing conditions. Is the lowest Rayleigh quotient of
the block-constant modes always $q\,\lambda_{1}(L(G))/\kappa+o(1)$, with $\kappa$ the eigenvalue of the
balancing vector as in Remark \ref{rem:linearmodel}, and for which graphs $G$ and sign structures can
such surfaces exist? The desingularisations of intersecting Clifford tori of \cite{KWlow}, in which
several sheets meet along a common locus, are natural test cases; for several of these families the
surfaces are determined by their symmetries and topology \cite{KWlaw}, so that the combinatorial
datum $G$ is intrinsic.
\end{problem}

\subsection{First eigenfunctions}
By Takahashi's theorem \cite{Takahashi}, Yau's conjecture implies that a minimal embedding of a closed surface in $S^3$ is by first eigenfunctions, and hence yields area bounds of the type of \cite{YangYau}. Theorem A places the surfaces
of \cite{Wiygul} among the known minimal embeddings by first eigenfunctions, a class enlarged considerably by the equivariant eigenvalue optimisation of \cite{KKMS}. Our argument gives the eigenvalue but not the eigenspace.

\begin{problem}
Is the first eigenspace of a stacked Clifford torus spanned by the restrictions of the coordinate
functions of $\bR^{4}$?
\end{problem}

The argument of Section \ref{sec:sector} shows that no $\cG$-invariant first eigenfunction exists;
a positive answer would require controlling the sectors of the other irreducible representations of
$\cG$, on which the reflection lemma gives no information.

\section*{Declarations}

\subsection*{Conflict of interest}

The author declares no conflict of interest.

\subsection*{Data availability}

Data sharing is not applicable to this article as no datasets were generated or analysed during the current study.

\end{document}